\documentclass[12pt]{article}

\usepackage{graphics,graphicx,fullpage,natbib,multirow}
\usepackage{amsmath,amssymb,verbatim,epsfig}
\usepackage[dvipsnames,usenames]{color}

\newtheorem{proposition}{Proposition}

\newtheorem{defin}{\bf Definition}

\newenvironment{proof}{\noindent{\bf Proof.}}{\hfill $\diamond$} 

\def\ga{\mbox{Ga}}
\def\iga{\mbox{IGa}}

\def\be{\mbox{Be}}
\def\ibe{\mbox{IBe}}
\def\ipa{\mbox{IPa}}
\def\bin{\mbox{Bin}}

\def\ghs{\mbox{GHS}}
\def\no{\mbox{N}}
\def\nbi{\mbox{NB}}
\def\pa{\mbox{Pa}}
\def\po{\mbox{Po}}

\def\gsst{\mbox{GSSt}}
\def\un{\mbox{Un}}

\def\E{\mbox{E}}
\def\V{\mbox{Var}}

\def\Cor{\mbox{Corr}}

\def\v{\mbox{V}}

\def\Cv{\mbox{Cov}}

\def\d{\mbox{d}}

\def\bn{{\bf n}}
\def\bs{{\bf s}}

\def\bS{{\bf S}}

\def\bY{{\bf Y}}

\def\simiid{\stackrel{\mbox{\scriptsize{iid}}}{\sim}}

\newcommand{\btheta}{\boldsymbol{\theta}}

\newcommand{\bomega}{\boldsymbol{\omega}}

\newcommand{\CC}{\mathcal{C}}

\newcommand{\FC}{\mathcal{F}}

\newcommand{\MC}{\mathcal{M}}
\newcommand{\PC}{\mathcal{P}}

\newcommand{\RR}{\mathbb{R}}

\newcommand{\rd}{\mathrm{d}}
\newcommand{\rr}{\mathbb{R}}

\newcommand{\ca}[1]{ {\cal #1} }

\newcommand\ie{{i.e., \/}}

\begin{document}

\baselineskip=24pt

\title{\bf General dependence structures for some models based on exponential families with quadratic variance functions}
\author{{\sc Luis Nieto-Barajas$^1$ \& Eduardo Guti\'errez-Pe\~na$^2$} \\[2mm]
{\sl $^1$ Department of Statistics, ITAM, Mexico} \\[2mm]
{\sl $^1$ Department of Probability and Statistics, IIMAS-UNAM, Mexico} \\[2mm]
{\small {\tt lnieto@itam.mx {\rm and} eduardo@sigma.iimas.unam.mx}} \\}
\date{}
\maketitle

\begin{abstract}
We describe a procedure to introduce general dependence structures on a set of random variables. These include order-$q$ moving average-type structures, as well as seasonal, periodic, spatial and spatio-temporal dependences. The invariant marginal distribution can be in any family that is conjugate to an exponential family with quadratic variance function. Dependence is induced via a set of suitable latent variables whose conditional distribution mirrors the sampling distribution in a Bayesian conjugate analysis of such exponential families. We obtain strict stationarity as a special case. 
\end{abstract}

\vspace{1ex} 


\section{Introduction}
\label{sec:intro}

Dependence structures on a set of random variables $\{Y_i\}$ can be modelled in various different ways. For instance, in the analysis of time series, autoregressive (AR) processes of order $p$ establish a linear dependence of the form 
\begin{equation}
\label{eq:ar}
Y_i = \beta_0 + \beta_1 Y_{i-1} + \cdots + \beta_p Y_{i-p} + Z_i,
\end{equation}
where $Z_i$ is a random error such that $\E(Z_i)=0$ and $\V(Z_i)=\sigma^2$, the $\beta_i$'s are coefficients, and the index $i$ is usually referred to time. Further conditions on the coefficients need to be imposed if we require the process \eqref{eq:ar} to be stationary \citep{box&jenkins:70}. If the number of elements in the set $\{Y_i\}$ is finite, conditions for \eqref{eq:ar} to be strictly stationary for any $p>0$ are not available. They do exist for $p=1$, but only in the normal case \cite[e.g.][]{mendoza&nieto:06}. 

On the other hand, moving average (MA) processes of order $q$ are defined through averages of the lagged errors as
\begin{equation}
\label{eq:ma}
Y_i = \gamma_0 + Z_{i} + \gamma_1 Z_{i-1} + \cdots + \gamma_q Z_{i-q},
\end{equation}
where the $\gamma_i$'s are coefficients. In this case, the process is always second-order stationary but no general conditions are known for strict stationarity, regardless of whether the set $\{Y_i\}$ is finite or not. 

For spatial models, the most common dependence structure for a set of random variables $\{Y_i\}$ is induced through conditional autoregressive (CAR) specifications \citep{besag:74}. Let $\partial_i$ be the set of neighbours of region $i$; then a CAR model is defined by 
\begin{equation}
\label{eq:car}
Y_i=\sum_{j\in\partial_i}\beta_{ij}Y_{j}+Z_i,
\end{equation}
where, as before, $Z_i$ is a random error such that $\E(Z_i)=0$ and $\V(Z_i)=\sigma_i^2$, and the $\beta_{ij}$'s are coefficients. Here, if the number of elements in the set $\{Y_i\}$ is finite, and
under the assumption that the errors are independent and follow a normal distribution, a symmetry condition on the $\beta_{ij}$'s and the $\sigma_i^2$'s is required to obtain the joint distribution of the $Y_i$'s. However, this joint distribution turns out to be improper. To achieve propriety, \ie a well-defined multivariate normal distribution, an extra parameter with a constrained support is required. Strict second-order stationarity is not possible even in this latter case. 

Autoregressive-type processes with marginal distributions other than normal have been studied for the case $p=1$. \cite{lawrance:82} and \cite{walker:00} studied the distribution of the innovation term $Z_i$ that ensures a gamma marginal distribution. In turn, \cite{nieto&walker:02} defined order-one dependent stationary processes with beta and gamma marginals by means of a Markov construction via latent variables. This construction was further generalized by \cite{pitt&al:02} to more general exponential families. 

Conditional autoregressive processes for spatial modelling were also considered for zero-one random variables by \cite{besag:74} himself, who proposed the so-called autologistic model. Extension to multinomial responses yields the Potts model \citep[e.g.][]{green&richardson:02}. In these two non-normal models, the joint distributions are not explicitly known since the normalising constants are not analytically available. Therefore stationarity conditions are out of the question. More recently, \cite{nieto&bandyopadhyay:13} proposed a Markov random field with gamma marginal distributions. 

The aim of this paper is to propose a way of constructing dependence structures on a collection of random variables for a class of distributions that are conjugate to exponential families with quadratic variance functions. The marginal distribution is invariant and dependences include order-$q$ moving average-type, seasonal, periodic, spatial and spatio-temporal. While our construction can in principle be applied to more general exponential families, those having a quadratic variance function allow a particularly neat treatment because their first and second moments can be calculated in closed form.

In the next section, we review the theory of exponential families. In Section~\ref{sec:model} we describe our construction and discuss its main properties. We then illustrate our model in Section~\ref{sec:data}, both with synthetic and real data sets. Section~\ref{sec:conclusions} contains some concluding remarks.

\section{Preliminaries}
\label{sec:expfam}

\subsection{Exponential families}

We start by reviewing some basic concepts concerning exponential family models and introducing some notation. The reader is referred to \cite{BN:78} for a comprehensive account of the theory of exponential families.

Let $\nu$ be a $\sigma$-finite measure on the Borel sets of
$\RR$ (typically Lebesgue measure or a counting measure) and consider the family $\PC = \{ P_{\omega} : \bomega \in \Omega \subseteq \RR^k \}$ of probability measures dominated by $\nu$ such that
\begin{equation*}
	\frac{ d P_{\omega}(x) }{ d\nu } =
	a(x) \, \exp \{ \btheta(\bomega)'\bs(x) -
	M(\btheta(\bomega)) \},
\end{equation*}
where $k$ is a positive integer, $\btheta(\cdot)=(\theta_1(\cdot),\ldots,\theta_k(\cdot))$, and $M(\cdot)$, $\theta_1(\cdot),\ldots,\theta_k(\cdot)$ are all real-valued functions. Also, $a(\cdot)$ is a nonnegative measurable function and $\bs(\cdot) = (s_{1}(\cdot),\ldots,s_{k}(\cdot))$ is some
measurable vector function. This is the general form of an exponential family and all the exponential family densities that appear in the remainder of this paper are particular cases of this expression.

Now assume that the mapping $\bomega \rightarrow \btheta(\bomega)$ is one-to-one and let 
$\FC = \{ f(x \mid \btheta) : \btheta \in \Theta \}$, where
\begin{equation}\label{eq:efam}
	f(x \mid \btheta) = 
	a(x)\,\exp\{ \btheta'\bs(x) - M(\btheta) \},
\end{equation}
with $M(\btheta) = \log \int a(x)\,\exp\{ \btheta'\bs(x) \}
\nu(\d x)$ and $\Theta = \mbox{int } \Xi$, where $\Xi =
\{ \btheta \in \RR^{k} : M(\btheta) < \infty \}$.
Then $\FC$ is an {\em exponential family} with \emph{canonical parameter} $\btheta$ and \emph{canonical statistic} $\bs(x)$. The set $\Theta$ is called the \emph{canonical parameter space}. The canonical statistic $\bs(x)$ is sufficient for the family $\FC$. Under certain regularity conditions, exponential families are essentially the only models admitting a sufficient statistic of a fixed finite dimension \citep{koop:36}. 

Here, we will mostly be concerned with situations where $k = 1$. If $s(\cdot)$ is the identity mapping,  $\FC$ is said to be a \emph{natural exponential family} \citep{morris:82} and $M(\cdot)$ is then called the \emph{cumulant transform} of $\FC$ \citep{BN:78}.
The exponential family $\FC$ is called \emph{regular} if $\Xi$ is an open subset of $\RR^{k}$ \citep{BN:78}. Hereafter, all the exponential families we will work with will be assumed to be regular.

Given a sample $X_1,X_2,\ldots,X_n$ of i.i.d.\ observations from $	f(x\mid\theta)$, the family of distributions of the corresponding sufficient statistic $S_n = \sum_{i=1}^{n} s(x_i)$ is a natural exponential family, with densities of the form
$f(s_n \mid \theta,n) = b(s_n,n)\,\exp\{ \theta s_n - n M(\theta) \}$.
%
Suppose now that $S$ is a random variable distributed according to 
$f(s_n \mid \theta,n)$ with $n=1$. The cumulants of $S$ can then be obtained by differentiating the cumulant transform $M(\theta)$. In particular,
$$
\E(S\mid\theta) = \frac{\d M(\theta)}{\d \theta} \qquad \mbox{ and } \qquad
\V(S \mid \theta) = \frac{\d^{2} M(\theta)}{\d \theta^2}.
$$

Consider the transformation $\mu(\theta) = E(S \mid \theta)$ and let $\mathcal{M} = \mu(\Theta)$. Since $M(\cdot)$ is convex, its second derivative is positive for all $\theta \in \Theta$ and $\mu = \mu(\theta)$ is a one-to-one transformation of~$\theta$. Thus $\mu$ provides an alternative parametrization of the family $\FC$, called the \emph{mean parametrization}. As we will see below, it is convenient to work in terms of this parametrization (instead of the canonical parametrization) since $\mu$ is the expected value of the observations and hence has a useful interpretation which is consistent across all exponential families. In this case, 
$f(s_n \mid \theta,n)$ can be written as
\begin{equation}\label{eq:efammp}
f(s_n | \theta(\mu),n) = b(s_n,n) \, 
\exp \{ \theta(\mu) s_n - n M(\theta(\mu)) \},
\end{equation}
where $\theta(\cdot)$ denotes the inverse of the transformation $\mu(\cdot)$.

Now let $V(\mu) = \V(S \mid \theta(\mu))$ for all $\mu \in \mathcal{M}$. This function is positive and is called the \emph{variance function} of the family $\FC$. The variance function, together with its domain $\mathcal{M}$, characterizes the family $\ca{F}$ within the class of all natural exponential families \citep{morris:82}.

\subsection{Conjugate families}
\label{sec:conjugacy}

A convenient way of modelling prior knowledge about $\theta$ is to use conjugate families. Since their introduction by \cite{RyS}, conjugate families of distributions have played an important role in Bayesian parametric inference. The main property of these families is that they are closed under sampling, but they often provide prior distributions which are tractable in various other respects. 

Let $\CC$ be the family of distributions defined on the Borel sets of $\Theta$ and with density function (with  respect to the Lebesgue measure) of the form
\begin{equation}\label{eq:cfam}
p(\theta \mid s_{0},n_{0}) = h(s_{0},n_{0}) 
\exp\{\theta s_{0}-n_{0} M(\theta)\},
\end{equation}   
where $s_{0}\in\rr$ and $n_{0}\in\rr$ are such that
\begin{equation*}
h(s_{0},n_{0})=\Big\{\int\exp\{\theta s_{0}-n_{0}M(\theta)\}\rd\theta\Big\}^{-1}
\end{equation*}
is well defined. Then $\CC$ is conjugate for $\FC$, meaning that the posterior distribution of $\theta$ also belongs to this family.
$\CC$ is commonly known as the {\em standard conjugate family} of the exponential family $\FC$ \citep{diaconis&ylvisaker:79}. Note that $\CC$ is also an exponential family, in this case with canonical statistic $(\theta,-M(\theta))$ and canonical parameter $(s_{0},n_{0})$. See \cite{GP-S} for an overview of conjugate families for exponential families.

Note that the conjugate family  (induced by $\CC$) for the mean parameter $\mu$ has densities of the form
\begin{equation}\label{eq:cfammp}
p_{\mu}(\mu \mid s_{0},n_{0}) = h(s_{0},n_{0}) 
\exp\{\theta(\mu) s_{0} - n_{0} M(\theta(\mu))\} \, |J_{\theta}(\mu)|,
\end{equation}
where $J_{\theta}(\cdot)$ denotes the Jacobian of the transformation $\theta(\cdot)$.

Having observed the sample $X_1,X_1,\ldots,X_n$ of i.i.d.\ observations from $f(x\mid\theta(\mu))$, with sufficient statistic $S_n=\sum_{i=1}^n s(X_i)$, the likelihood function of $\mu$ is proportional to  \eqref{eq:efammp}. Assuming a prior distribution for $\mu$ with a conjugate density function of the form \eqref{eq:cfammp}, the posterior distribution then has a density function of the form \begin{equation}\label{eq:cfammppost}
p_{\mu}(\mu \mid s^*,n^*) = h(s^*,n^*) 
\exp\{\theta(\mu) s^* - n^* M(\theta(\mu))\} \, |J_{\theta}(\mu)|,
\end{equation}
where $s^* = s_{0} + s_n$ and $n^* = n_0 + n$.

We close this section with the following important result of \cite{diaconis&ylvisaker:79}. If $\FC$ is a regular natural exponential family, and the prior distribution of $\theta$ belongs to the corresponding standard conjugate family $\CC$ with $n_0 > 0$ and $s_0/n_0 \in \mathcal{M}$, then
$$
E(\mu \mid s_0, n_0) = \frac{s_0}{n_0}.
$$
This result implies that the posterior expectation of $\mu$, $E(\mu \mid s^*, n^*)$, is linear in the sample mean $\bar{x}$.

\subsection{The NEF-QVF class}
\label{sec:nef-qvf}

For many common families of distributions, the variance function takes a simple form. \cite{morris:82} characterized all natural exponential families having a quadratic variance function of the form $\v(\mu)=\nu_0+\nu_1\mu+\nu_2\mu^2$, and found that there exist only six families with such property: normal (with known variance), Poisson, gamma, binomial, negative binomial and generalized hyperbolic secant distributions. In a subsequent paper, \cite{morris:83} developed the statistical theory for this class of distributions. In particular, he described the corresponding conjugate families and proved that
$$\V(\mu \mid s_0, n_0) = \frac{\v(s_0/n_0)}{n_0 -\nu_2}.$$

Table \ref{tab:nef} provides some relevant characteristics of the standard member of each of these six families. We describe the notation here: $\no(\mu,\tau)$ denotes a normal distribution with mean $\mu$ and precision $\tau$; $\po(\mu)$ denotes a Poisson distribution with mean $\mu$; $\ga(\alpha,\beta)$ denotes a gamma distribution with mean $\alpha/\beta$; $\bin(p,m)$ denotes a binomial distribution with success probability $p$ and number of Bernoulli trials $m$; $\nbi(p,m)$ denotes a negative binomial distribution with success probability $p$ and number of failures $m$; $\ghs(\mu,\alpha)$ denotes a generalized hyperbolic secant distribution with mean $\mu$ and precision parameter $\alpha$; $\iga(\alpha,\beta)$ denotes an inverse gamma distribution with mean $\beta/(\alpha-1)$; $\be(\alpha,\beta)$ denotes a beta distribution with mean $\alpha/(\alpha+\beta)$; $\ibe(\alpha,\beta)$ denotes an inverse beta or beta of the second kind distribution with mean $\alpha/(\beta-1)$; $ \gsst(\mu,m)$ denotes a generalized scaled Student $t$ distribution with mean $\mu$ and precision parameter $m$. 

This class, hereafter denoted by NEF-QVF, contains some of the most widely used families of distributions in applied statistics. Our models will be based on this class.

\section{The model}
\label{sec:model}

\subsection{General framework}
\label{sec:genframe}

Let $\bY=\{Y_i\}$ be a set of variables of interest, and let $\bS=\{S_i\}$ and $U$ be latent variables. With these components, we will define a three-level hierarchical model to achieve dependencies within  $\bY$. The dependence can be temporal, as illustrated in Figure~\ref{fig:temporal}, and/or spatial, as illustrated in Figure~\ref{fig:spatial}. On top of this, we want to ensure that the marginal distributions of the $Y_i$'s are all the same and belong to a given exponential family, as described below.

The whole idea is based on the notion of conjugacy for the Bayesian analysis of exponential families, as described in Section~\ref{sec:conjugacy}. We summarize our construction as follows: first, the desired marginal distribution on the $Y_i$'s is assigned to the latent variable $U$, and plays the role of the (conjugate) prior distribution; second, the latent variables $S_i$'s are assumed to  be conditionally independent given $U$, and play the role of the observations that give rise to the likelihood function; finally, the variables of interest   in the set $\bY$ are assumed to be conditionally independent given the $S_i$'s, and the conditional distribution of each $Y_i$ takes the form of the corresponding posterior distribution.

Let $\partial_i$ be the set of ``neighbours'', in a broad sense, of unit $i$. For a temporal dependence model of order $q$, $\partial_i$ would  be described by the set of indexes $\{i-q,\ldots,i-1,i\}$, whereas in a spatial dependence model, $\partial_i$ would be the set of actual neighbours of any order plus the current unit $i$. For a seasonal dependence model of order $q$, if the seasonality of the data is $s$, $\partial_i$ would be given by the set of indexes $\{i-qs,\ldots,i-q,i\}$. For a periodic dependence model of orders $(q_1,\ldots,q_s)$, we define $i=i(r,m)=(r-1)s+m$ for $r=1,2,\ldots$ and $m=1,\ldots,s$. To be specific, for monthly data, $s=12$, $r$ denotes the year and $m$ denotes the month. In this case $\partial_i$ would be the set of indexes $\{i(r,m)-q_m,\ldots,i(r,m)-1,i(r,m)\}$. Spatio-temporal models can also be specified using a suitable definition of the set $\partial_i$ \citep[e.g.][]{nieto:20}. For instance, if $i$ denotes location and $t$ denotes time, $\partial_{i,t}=\{(i,s):s\in\{t-q,\ldots,t-1,t\}\}\cup\{(j,t):j\mbox{ is neighbour of }i\}$  would allow us to define a spatio-temporal dependence model. 

Let us start by choosing the desired marginal distribution of $Y_i$, which we take to be a member of the exponential family of the form (\ref{eq:cfammp}). We then assign this distribution to the latent variable $U$, that is \begin{equation}
\label{eq:u}
p_u(u \mid s_0,n_0)=h(s_0,n_0)
\exp\left\{\theta(u) s_0 -  n_0 M(\theta(u))\right\}
\left|D_\theta(u)\right|. 
\end{equation}

Now, conditional on $U$, the $S_i$'s are independent with densities of the form (\ref{eq:efammp}); that is
\begin{equation}
\label{eq:s}
f(s_i \mid \theta(u),n_i) = b(s_i,n_i)\exp\{\theta(u) s_i - n_iM(\theta(u))\},
\end{equation}
for $i=1,\ldots,m$. 

Finally, conditional on $\bS$, the $Y_i$'s are independent with an exponential family distribution of the form (\ref{eq:cfammppost}), \ie 
\begin{equation}
\label{eq:y}
p_y(y_i\mid s^*_i,n^*_i) = h(s^*_i,n^*_i) \, \exp\left\{\theta(y_i) s^*_i - 
n^*_i M(\theta(y_i))\right\} \left|D_\theta(y_i)\right|,
\end{equation}
for $i=1,\ldots,m$, with 
\begin{equation}
\label{eq:snstar}
s^*_i=s_0+\sum_{j\in \partial_i}s_j \quad\mbox{and}\quad n^*_i=n_0+\sum_{j\in \partial_i}n_j.
\end{equation}

We are now ready to state our main result.
\begin{proposition}
\label{prop:1}
Let $\{Y_i\}$ be a set of random variables whose distribution is characterized by equations \eqref{eq:u}--\eqref{eq:y}. Then, 
\begin{description}
\item[(i)] The marginal distribution of each $Y_i$ is the same as the distribution of $U$, and
$\E(Y_i) = s_0 / n_0$ for all $i$. 
\end{description}
Moreover, if the distribution of the latent variables $\bS$, given by \eqref{eq:s}, belongs to the NEF-QVF class, then
\begin{description}
\item[(ii)] $\mathrm{Var}(Y_i) = V(s_0 /n_0 ) / (n_0 - \nu_2 )$ for all $i$ and, for any pair of random variables (\,$Y_i$, $Y_k$), the correlation induced by the model is given by
$$
\mathrm{Corr}(Y_i,Y_k)=\frac{n_0\left(\sum_{j\in\partial_i\cap\partial_k}n_j\right) + \left(\sum_{j\in\partial_i} n_j\right) \left(\sum_{j\in\partial_k} n_j\right)}
{(n_0+\sum_{j\in \partial_i}n_j)(n_0+\sum_{j\in \partial_k}n_j)}
$$
if $i\neq k$, and $\mathrm{Corr}(Y_i,Y_k)=1$ if $i=k$. 
\end{description}
\end{proposition}
\begin{proof}
	
\noindent \textbf{(i)} Using the notation and definitions of Section~\ref{sec:expfam}, we first note that
$$
p_\mu(\mu \mid s_0,n_0) =  
\int p_\mu(\mu \mid s^*_i,n^*_i) \, f(s_n \mid s_0, n_0) \, \nu(\d s)
$$
where
$$
f(s_n \mid s_0, n_0) = 
\int f(s_n \mid \theta(\tilde{\mu}),n) \, p_\mu(\tilde{\mu} \mid s_0,n_0) \, \d \tilde{\mu}
$$
is the prior predictive density of $S_n$. Hence,
$$
p_\mu(\mu \mid s_0,n_0) =  \int p_\mu(\mu \mid s^*_i,n^*_i) \, 
\left\{ \int f(s_n \mid \theta(\tilde{\mu}),n) \, p_\mu(\tilde{\mu} \mid s_0,n_0) \, \d \tilde{\mu} \right\} \nu(\d s).
$$
Now, identifying $y_i$, $s_i$, $n_i$ and $u$ with $\mu$, $s_n$, $n$ and $\tilde{\mu}$, respectively, and using Fubini's theorem, this latter expression becomes
$$
p_{y_i}(y_i \mid s_0,n_0) =  \int \int p_{y_i}(y_i \mid s^*_i,n^*_i) \, 
f(s_i \mid \theta(\tilde{\mu}),n_i) \, p_u(u \mid s_0,n_0) \, \d u \, \nu(\d s),
$$
which is precisely the marginal distribution of $Y_i$ obtained from the joint distribution \linebreak
$p(y_i, s_i, u \mid n_i, s_0,n_0)$ implied by the hierarchical structure \eqref{eq:u}--\eqref{eq:y}.

This shows that the marginal distribution of $Y_i$ is the same as the distribution of $U$.
That $\E(Y_i) = s_0 / n_0$ then follows from the result of  \cite{diaconis&ylvisaker:79} mentioned at the end of Section~\ref{sec:conjugacy}.

\vspace{1ex}

\noindent \textbf{(ii)} It follows from the results of Section~\ref{sec:expfam} that 
$\E(S_i\mid u)=n_i u$, $\V(S_i\mid u)=n_i\v(u)$, $\E(U)=s_0/n_0$ and $\V(U)=\v(s_0/n_0)/(n_0-\nu_2)$. 
Note that $\E(Y_i\mid\bs)$ and $\V(Y_i\mid\bs)$ have the same form as the corresponding moments of $U$ but with $s_0$ and $n_0$ replaced by~$s_i^*$ and~$n_i^*$. 

We will compute the marginal covariance of $Y_i$ and $Y_k$ using the ``law of total covariance'', $\Cv(Y_i,Y_k)=\E\{\Cv(Y_i,Y_k\mid\bS)\}+\Cv\{\E(Y_i\mid\bS),\E(Y_k\mid\bS)\}$. Note that the first term is zero due to conditional independence, and so $\Cv(Y_i,Y_k)=\Cv(S_i^*/n_i^*,S_k^*/n_k^*)$. We split $\partial_i$ and $\partial_k$ into two disjoint sets, $\partial_i=(\partial_i\cap\partial_k)\cup(\partial_i-\partial_k)$ and $\partial_k=(\partial_i\cap\partial_k)\cup(\partial_k-\partial_i)$, such that $(\partial_i\cap\partial_k)\cap(\partial_i-\partial_k)\cap(\partial_k-\partial_i)=\emptyset$. We can thus rewrite the covariance as $\Cv(Y_i,Y_k) = \{\V(A)+\Cv(A,C)+\Cv(B,A)+\Cv(B,C)\} / (n_i^*n_k^*)$, where 
$$A=\sum_{j\in\partial_i\cap\partial_k}S_j,\quad B=\sum_{j\in\partial_i-\partial_k}S_j,\quad C=\sum_{j\in\partial_k-\partial_i}S_j.$$
We can now compute each of the elements of the previous sum separately, again using the ``law of total (co)variance''.

First, $\V(A)=\E\{\V(A\mid U)\}+\V\{\E(A\mid U)\}$, which after some algebra becomes $\V(A) = \left(\sum_{j\in\partial_i\cap\partial_k}n_j\right)\E\{\v(U)\}+\left(\sum_{j\in\partial_i\cap\partial_k}n_j\right)^{\!\! 2} \V(U)$. Substituting the form of the quadratic variance function and taking expectations,  after some algebra we obtain $\E\{\v(U)\}=n_0\V(U)$. On the other hand, $\Cv(A,C)=\E\{\Cv(A,C\mid U)\}+\Cv\{\E(A\mid U),\E(C\mid U)\}$. The first term vanishes due to conditional independence , so  $$\Cv(A,C)=\left(\sum_{j\in\partial_i\cap\partial_k}n_j\right)\left(\sum_{j\in\partial_k-\partial_i}n_j\right)\V(U).$$ The computations for $\Cv(B,A)$ and $\Cv(B,C)$ are analogous. Substituting all these expressions back in the expression for the covariance of $Y_i$ and $Y_k$, we obtain $\Cv(Y_i,Y_k)=\V(U)/(n_i^*n_k^*)\left\{n_0\left(\sum_{j\in\partial_i\cap\partial_k}n_j\right)+\right.$ $\left.\left(\sum_{j\in\partial_i}n_j\right)\left(\sum_{j\in\partial_k}n_j\right)\right\}$. Finally, upon recalling that $\V(U)=\V(Y_i)=\V(Y_k)$, we obtain the result. 
\end{proof}

\vspace{2ex}

Proposition \ref{prop:1} tells us that all the $Y_i$'s have the same marginal distribution, and therefore the same mean and variance. Additionally, the correlation between any two variables, say $Y_i$ and $Y_k$ with $i\neq k$, is fully characterized and it has a simple expression. This correlation can be split into two parts: first, a function of the marginal parameter $n_0$ and the shared parameters $n_j$ in the sets $\partial_i$ and $\partial_k$; and second, a function of the two whole sets of parameters $n_j$ in the sets $\partial_i$ and $\partial_k$. The sum of these two components is then normalized so it lies in the interval $[0,1]$. Moreover, if $n_i=n$ for all $i$, then the set $\bY$ becomes strictly stationary.

\subsection{Models based on the NEF-QVF class}
\label{sec:examples}

In what follows we describe six families of models. The name given to each of these models has the form ``Family~1 - Family~2'', with Family~1 indicating the marginal distribution of $Y_i$ and Family 2 referring to the associated NEF-QVF used for the latent variables $S_i$. All conjugate families are characterised by equations \eqref{eq:u}, \eqref{eq:s} and \eqref{eq:y}. We provide the specific forms of the functions involved in each case. Expressions for $s_i^*$ and $n_i^*$, $i=1,\ldots,m$ are given in \eqref{eq:snstar}. 

\bigskip\noindent\textbf{Normal - Normal}

\smallskip\noindent
This family is characterised by the functions
$b(s_i,n_i)=\left(2\pi n_i\right)^{-1/2}\exp\left\{-{s_i^2}/{(2n_i)}\right\}$, $h(s_0,n_0)=\left(2\pi/n_0\right)^{-1/2}\exp\left\{-s_0^2/(2n_0)\right\}$, $\theta(u)=u$, $M(\theta(u))=u^2/2$ and $D_\theta(u)=1$.
To construct a process $\bY$ with normal marginal distributions, we define our three-level hierarchical model for $i=1,\ldots,m$ as 
\begin{equation}
\label{eq:normalm}
U\sim\no(s_0/n_0,n_0),\quad [S_i\mid U] \sim \no(n_i u,1/n_i)\quad\mbox{and }\quad [Y_i\mid\bS] \sim \no(s_i^*/n_i^*, n_i^*). 
\end{equation}

\bigskip\noindent\textbf{Gamma - Poisson}

\smallskip\noindent
In this case, $h(s_0,n_0)=n_0^{s_0}/\Gamma(s_0)$, $\theta(u)=\log(u)$, $M(\theta(u))=u$, $D_\theta(u)={1}/{u}$ and $b(s_i,n_i)=n_i^{s_i}/s_i!$. 
A process $\bY$ with gamma marginal distributions can be defined for $i=1,\ldots,m$ as follows:
\begin{equation}
\label{eq:gamma1m}
U\sim\ga(s_0,n_0),\quad [S_i\mid U] \sim \po(n_i u)\quad\mbox{and}\quad [Y_i\mid\bS] \sim \ga(s_i^*, n_i^*). 
\end{equation}

\bigskip\noindent\textbf{Inverse Gamma - Gamma}

\smallskip\noindent
Now $h(s_0,n_0)=s_0^{n_0}/\Gamma(n_0)$, $\theta(u)=-{1}/{u}$, $M(\theta(u))=\log(u)$,  $D_\theta(u)={1}/{u}$ and $b(s_i,n_i)=s_i^{n_i-1}/\Gamma(n_i)$.
A process $\bY$ with inverse gamma marginal distributions would be defined by means of the following hierarchical specification for $i=1,\ldots,m$:
\begin{equation}
\label{eq:igammam}
U\sim\iga(n_0+1,s_0),\quad [S_i \mid U] \sim \ga(n_i,1/u)\quad\mbox{and} \quad 
[Y_i \mid \bS] \sim \iga(n_i^*+1, s_i^*).
\end{equation}

\bigskip\noindent\textbf{Beta - Binomial}

\smallskip\noindent
This family is characterised by the functions $h(s_0,n_0)=\Gamma(n_0)/\{\Gamma(s_0)\Gamma(n_0-s_0)\}$, $\theta(u)=\log\left\{{u}/{(1-u)}\right\}$, $M(\theta(u))=-\log(1-u)$, $D_\theta(u)={1}/{\{u(1-u)\}}$ and $b(s_i,n_i)={n_i\choose s_i}$.
If we want to construct a process $\bY$ with beta marginal distributions we can define the three-level hierarchical model for $i=1,\ldots,m$ as
\begin{equation}
\label{eq:beta1m}
U\sim\be(s_0,n_0-s_0),\quad [S_i \mid U] \sim \bin(u,n_i) \quad \mbox{and} \quad 
[Y_i \mid \bS] \sim \be(s_i^*,n_i^*-s_i^*). 
\end{equation}

\bigskip\noindent\textbf{Inverse Beta - Negative Binomial}

\smallskip\noindent
Here, $h(s_0,n_0)=\Gamma(s_0+n_0+1)/\{\Gamma(s_0)\Gamma(n_0+1)\}$, $\theta(u)=\log\left\{{u}/{(u+1)}\right\}$, $M(\theta(u))=\log(u+1)$, $D_\theta(u)={1}/{\{u(u+1)\}}$ and $b(s_i,n_i)={n_i+s_i-1\choose s_i}$. 
To construct a process $\bY$ with inverse beta marginal distributions we can define the three-level hierarchical model for $i=1,\ldots,m$ as 
\begin{equation}
\label{eq:ibetam}
U \sim \ibe(s_0,n_0+1),\quad [S_i\mid U] \sim \nbi(u/(u+1),n_i)\quad\mbox{and}\quad [Y_i\mid\bS] \sim \ibe(s_i^*,n_i^*+1). 
\end{equation}

\bigskip\noindent\textbf{Generalized Scaled Student - Generalized Hyperbolic Secant}

\smallskip\noindent
Finally, in this case $h(s_0,n_0)$ is a normalising constant, $\theta(u)=\tan^{-1}(u)$, $M(\theta(u))=({1}/{2})\log\left(1+u^2\right)$, $D_\theta(u)=\left(1+u^2\right)^{-1}$ and $b(s_i,n_i)=\{{2^{n_i-2}}/{\Gamma(n_i)}\} \prod_{j=0}^\infty\left\{1+{s_i^2}/{(n_i+2j)^2}\right\}^{-1}$. 
If we want to construct a process $\bY$ with generalized scaled Student $t$ marginal distributions, we need to specify the following three-level hierarchical model for $i=1,\ldots,m$:
\begin{equation}
\label{eq:sstm}
U\sim \gsst(s_0/n_0,n_0),\quad [S_i\mid U] \sim \ghs(n_i u,1/n_i)\quad\mbox{and }\quad [Y_i\mid\bS] \sim  \gsst(s_i^*/n_i^*, n_i^*). 
\end{equation}

\section{Numerical analyses}
\label{sec:data}

\subsection{Synthetic data}

In order to show how the correlation induced by our construction may look like, we will use the temporal dependence structure of order $q=2$ described by Figure~\ref{fig:temporal} and the simple spatial dependence structure described by Figure~\ref{fig:spatial}. 

For the temporal model, the neighbouring sets are $\partial_1=\{1\}$, $\partial_2=\{1,2\}$, and $\partial_i=\{i,i-1,i-2\}$ for $i=3,4\ldots$. We computed $\Cor(Y_i,Y_k)$ for $i=1$ and $k=1,2,\ldots,16$ and considered several scenarios for the parameters $n_j$, $j=0,1,\ldots$. Specifically, we take fixed values $n_j=1$, for $j=1,2,\ldots$ and vary $n_0\in\{0.01,0.1,1,10\}$. The correlations induced by these scenarios are shown in the left-hand panel of Figure~\ref{fig:tcor}. The correlation starts at one, for $k=1$, and decreases to a constant value, for $k=2,3$, and decreases again to a lower constant value for $k=4,5,\ldots$. The effect of $n_0$ is inverse, the correlation decreases as $n_0$ increases. In the same panel we also include a scenario that shows that the correlation is not necessarily monotonic; this can be achieved, say, by setting $n_0=0.5$ and varying values of $n_j$ for $j=1,2,\ldots$. The right-hand panel of Figure~\ref{fig:tcor} shows five paths obtained with simulated values of the parameters taken from $n_0\sim\un(0,2)$, a uniform distribution on the interval (0,2), and $[ n_j\mid a,b ]` \sim \ga(a,b)$ for $j=1,\ldots,16$, with $a\sim\ga(1,1)$ and $b\sim\ga(1,1)$. These paths illustrate the flexibility of the correlation induced by our construction. 

For the spatial model, the neighbouring sets are $\partial_1=\{1,2,3\}$, $\partial_2=\{1,2,3\}$, $\partial_3=\{1,2,3,4,5\}$, $\partial_4=\{3,4,5\}$ and $\partial_5=\{3,4,5\}$. Again, we computed $\Cor(Y_i,Y_k)$ for $i,k=1,2,\ldots,5,$ and considered several scenarios for values of the parameters. The $5\times 5=25$ correlations are presented as vertical lines in Figures~\ref{fig:scor1} and~\ref{fig:scor2}, where, for each value of $k=1,\ldots,5$ on the horizontal axis, we include five vertical lines which correspond to $i=1,\ldots,5$ and are shown as 5 line types. 
In Figure \ref{fig:scor1} we took $n_j=1$ for $j=1,\ldots,5$, and three values $n_0\in\{0.1,1,10\}$ shown in the top-left, top-right and bottom-left panels, respectively. In all cases, region $k=3$ is the one with higher correlations, because it is the one with larger number of neighbours. Larger values of $n_0$ imply smaller correlations. The bottom-right panel illustrates the case where $n_0=1$ and different $n_j$'s for $j>0$. Figure \ref{fig:scor2} includes four scenarios, one in each panel, obtained with simulated values of the parameters taken from $n_0\sim\un(0,2)$ and $[ n_j \mid a,b ] \sim \ga(a,b)$ for $j=1,\ldots,5$, with $a\sim\ga(1,1)$ and $b\sim\ga(1,1)$. Again, the highest correlations involve region $k=3$; however, other regions also show large correlations.

\subsection{Real life application}

Mexico's National Institute of Statistics and Geography (INEGI) carries out a nation-wide continuous survey to study the occupation and employment of the population. They define the unemployment rate as the percentage of people, 15 years or older, who are looking for a job and have not been able to find it; rates are provided in percentages in a scale from 0 to 100. To illustrate the performance of the dependence constructions proposed, we will carry out two analyses: a temporal and a spatial analysis. 

For the temporal analysis, we consider the monthly unemployment rates ($Y_i$) from January 2006 to September 2019, that is, for $i=1,\ldots,165$. The data are shown as solid lines in the second to fourth panels of Figure~\ref{fig:dataT}. From 2009 to 2016 there was a period of high rates. We propose to model the data with an inverse gamma - gamma model; that is, a temporal dependence model with an $\iga(\alpha,\beta)$ marginal distribution, and with latent variables $S_i$ coming from a gamma distribution as in \eqref{eq:igammam}. Temporal dependence was defined with neighbouring sets $\partial_i=\{i-q,\ldots,i-1,i\}$, and with $q\in\{0,1,\ldots,14\}$ for comparison purposes. We implemented a Bayesian analysis with independent prior distributions $\alpha\sim\ga(0.1,0.1)$ and $\beta\sim\ga(0.1,0.1)$; the hierarchical prior for the $n_i$'s was defined as $n_i\mid n_0\simiid\ga(1,n_0)$ and $n_0\sim\ga(1,1)$. 

Posterior inferences were obtained via a Gibbs sampler \citep{GG:84}. We ran two chains with 15,000 iterations, a burn-in period of 5,000 iterations, and a thinning of 5 iterations to reduce the autocorrelation within each chain. The convergence of the chains with these specifications was satisfactory according to the trace and ergodic means plots. Since the involved distributions that define our model are all of standard form, inference can be carried out in {\sf R} through the JAGS software \citep{plummer:18}. 

We assessed the model fit by computing the deviance information criterion (DIC) originally introduced by \cite{spiegelhalter&al:02}. Smaller values of DIC are preferred. These values for varying $q\in\{1,\ldots,14\}$ are shown in the top-left panel of Figure~\ref{fig:dataT}. The largest improvement in the DIC scale is given from the model with $q=0$ (exchageability) to the model with $q=1$. The values continue to fall until $q=11$ (best model) and then increase for $q\geq 12$. 

The impact on the fit for some choices of $q$ can be seen in the second to fourth panels of Figure~\ref{fig:dataT}. Point predictions (thick solid lines) and 95\% credible intervals (dotted lines) are shown. For $q=0$ the $Y_i$'s become exchangeable with the same marginal distribution so the prediction is constant, whereas for a temporal dependence of order $q=1$, the predictions follow the path of the data closely but with a somewhat larger uncertainty (wider credible intervals). However, for the best-fitting model, obtained with $q=11$, the point predictions are smoother and with considerably less uncertainty (narrower credible intervals). 

To place the performance of our model in context, we also fitted a dynamic linear model \citep{west&harrison:97} of the form $Y_i\sim\no(\mu_i,\tau)$ and $\mu_i\sim\no(\mu_{i-1},\tau_\mu)$ for $i=1,\ldots,165$ with prior distributions $\mu_0\sim\no(0,0.1)$, $\tau\sim\ga(0.1,0.1)$ and $\tau_\mu\sim\ga(0.1,0.1)$. The value of the DIC obtained is $-110$, which is higher than the corresponding values obtained with our model, which range from $-950$ to $-550$. 

For the spatial analysis we consider the unemployment rates ($Y_i$) for all Mexican States in the fourth trimester of 2019; that is, $i=1,\ldots,32$. The data are shown in the top panel of Figure~\ref{fig:dataS}. The states with higher unemployment rates are Baja California Sur (ID-3), Chihuahua (ID-5) and Sonora (ID-26) in the north, Edomex (ID-15) and Queretaro (ID-22) in the center, and Tabasco (ID-27) in the south. We propose to use the same model as above, the inverse gamma - gamma model, but with a spatial dependence structure, $\partial_i$, given by the actual neighbouring states. A Bayesian analysis was implemented with the same prior distributions and with the same MCMC settings as those used for the temporal data. 

In the bottom panel of Figure~\ref{fig:dataS} we show posterior predictive estimates. Compared to the actual data (top panel), our predictions are smoother as a consequence of the spatial dependence. The new map clearly shows clustered regions, with the north-west part of Mexico showing higher unemployment rates than the south-centre region. From Chiapas (ID-7) and Tabasco (ID-27) towards the Yucatan peninsula, the predicted rates remain almost the same as the observed rates. 

Finally, for the sake of comparison, we also fitted a normal CAR model. This is defined as $Y_i\mid\bY_{-i}\sim\no\left(\rho\sum_{j=1}^nI(i\smile j)y_j/r_i,r_i\,\tau\right)$ with $r_i$ the number of neighbours in each state $i=1,\ldots,32$, with prior distributions for the association parameter $\rho\sim\un(0,1)$ and for the precision $\tau\sim\ga(0.1,0.1)$. The DIC values obtained with our spatial model and with the CAR model are almost the same, $-45.9$ and $-45.6$, respectively. However, the predicted map produced with the CAR model is considerably different to the one produced with our spatial model. 

Regarding computational times, our temporal dependence model took 45 seconds each run as compared to the dynamic linear model which took only 5 seconds to run. The spatial dependence model took 3.5 minutes, whereas the CAR model took 5 seconds to run. Analysis were run in an Intel Core i7 with 16 GB of RAM. The data and code for these analyses and the fittings obtained for the competitors (dynamic and CAR models) are provided as Supplementary Material.

\section{Concluding remarks}
\label{sec:conclusions}

In this paper, we have described a general procedure to introduce dependence structures on a set of random variables whose distribution belongs to a class of models which are conjugate to natural exponential families with quadratic variance functions. Such dependence is induced by means of a set of latent variables within a three-level hierarchical model, and is fully characterized via the correlation function. 

This general construction is based on the well-known Bayesian notion of conjugacy. We used the conjugate prior distribution of the mean parameter in the first level of the hierarchy, a sampling model in the NEF-QVF class for the second level, and the corresponding (conjugate) posterior in the third level. This allows a particularly neat treatment because all the required first and second moments can be calculated in closed form.

While we focused on the important NEF-QVF class written in terms of the mean parametrisation (a class that contains some of the most commonly used models), the same conjugate construction can be used with other parametrisations and other exponential families. Consider, for instance, the inverse gamma - gamma model \eqref{eq:gamma1m}. If we took $U'=1/U$ and $Y_i'=1/Y_i$, then the model would be rewritten as 
$$U'\sim\ga(n_0+1,s_0),\quad [S_i\mid U'] \sim \ga(n_i,u')\quad\mbox{and}\quad [Y_i'\mid\bS] \sim \ga(n_i^*+1, s_i^*).$$
As another example, the inverse beta-negative binomial model \eqref{eq:ibetam} could be similarly rewritten if we applied the same transformations as above to obtain
$$U'\sim\be(n_0+1,s_0),\quad [S_i\mid U'] \sim \nbi(n_i,u')\quad\mbox{and}\quad [Y_i'\mid\bS] \sim \be(n_i^*+1,s_i^*).$$
These transformations yield alternative representations of the marginal gamma and marginal beta constructions, respectively. However, in these cases the correlation between any two random variables may not be available in closed form.

As a matter of fact, the procedure described in this paper can also be applied beyond exponential family settings, as long as a conjugate structure is available. For example, \cite{nieto&huerta:17} defined a spatio-temporal model with Pareto marginal distributions based on the conjugacy between the Pareto and inverse Pareto distributions. In particular they took $U\sim\pa(n_0,s_0)$, $[S_i\mid U]\sim\ipa(n_i,1/u)$ and $[Y_i\mid\bS]\sim\pa(n_i^*,s_i^*)$ with $n_i^*$ as in \eqref{eq:snstar} and $s_i^*=\max\left(s_0,\max_{j\in\partial_i}\{s_j\}\right)$. The correlation induced by this construction has a similar behaviour to the one shown here and can be computed in closed form. 

Maintaining a desired marginal distribution in a set of dependent variables is useful in applications where a particular model has some desirable features. For instance, the Dirichlet process can be defined through a sequence of beta random variables. To define dependent Dirichlet processes, we could use the beta construction \eqref{eq:beta1m} to maintain the same marginal distribution and introduce dependence \citep[e.g.][]{nieto&al:12}. 

Finally, we note that both simulating these processes and obtaining posterior samples for Bayesian
inference on the parameters of the models can, in most cases, be carried out without much difficulty due to their hierarchical and conjugate structure.

\section*{Acknowledgements} 

The authors would like to thank two anonymous reviewers whose detailed comments helped to improve the presentation of this paper. Most of this work was carried out while the second author was visiting the Department of Statistics at ITAM. He is very grateful to this institution for their hospitality. Partial support from Mexico's \textit{Sistema Nacional de Investigadores} is gratefully acknowledged. The first author was also supported by \textit{Asociaci\'on Mexicana de Cultura, A.C.}

\bibliographystyle{natbib}

\vspace{10ex}

\begin{table}[h]
\centering
{\small
\begin{tabular}{lllll} \hline\hline
\textbf{Family} & \textbf{Notation} &  \textbf{Variance fun.} & \textbf{Domain} & \textbf{Conjugate fam.} \\ \hline
Normal & $\no(\mu,1)$ & $V(\mu) = 1$ & $\MC = \RR$ & $\no(s_0/n_0,n_0)$ \\
Poisson & $\mbox{Po}(\mu)$ & $V(\mu) = \mu$ & $\MC = \RR^{+}$ & $\ga(s_0,n_0)$ \\
Gamma & $\ga(1,1/\mu)$ & $V(\mu) = \mu^{2}$ & $\MC = \RR^{+}$ & $\iga(n_0+1,s_0)$ \\
Binomial & $\bin(\mu,1)$ & $V(\mu) = \mu-\mu^2$ & $\MC = (0,1)$ & $\be(s_0,n_0-s_0)$ \\
Neg. Bin. & $\nbi(\mu/(1+\mu),1)$ & $V(\mu) = \mu+\mu^2$ & $\mathcal{M} = \RR^{+}$ & 
$\ibe(s_0,n_0+1)$ \\
Hyper. Sec. & $\ghs(\mu,1)$ & $V(\mu)=1+\mu^{2}$ & $\MC = \RR$ & $ \gsst(s_0/n_0,n_0)$ \\
\hline\hline
\end{tabular}}
\caption{The NEF-QVF class, notation and relevant characteristics.}
\label{tab:nef}
\end{table}

\newpage
\bigskip

\begin{figure}[h]
\setlength{\unitlength}{0.8cm}
\begin{center}
\begin{picture}(20,8)
\put(9.8,7.6){$U$} 
\put(10,7.3){\vector(-2,-1){5}}
\put(10,7.3){\vector(-1,-1){2.5}}
\put(10,7.3){\vector(0,-1){2.5}}
\put(10,7.3){\vector(1,-1){2.5}}
\put(10,7.3){\vector(2,-1){5}}
\put(4.6,4.1){$S_1$}
\put(7.2,4.1){$S_2$}
\put(9.8,4.1){$S_3$}
\put(12.4,4.1){$S_4$}
\put(15,4.1){$S_5$}
\put(4.6,1.1){$Y_1$}
\put(7.2,1.1){$Y_2$}
\put(9.8,1.1){$Y_3$}
\put(12.4,1.1){$Y_4$}
\put(15,1.1){$Y_5$}
\put(4.8,3.9){\vector(0,-1){2.2}}
\put(4.8,3.9){\vector(1,-1){2.3}}
\put(4.8,3.9){\vector(2,-1){4.8}}
\put(7.4,3.9){\vector(0,-1){2.2}}
\put(7.4,3.9){\vector(1,-1){2.3}}
\put(7.4,3.9){\vector(2,-1){4.8}}
\put(10,3.9){\vector(0,-1){2.2}}
\put(10,3.9){\vector(1,-1){2.3}}
\put(10,3.9){\vector(2,-1){4.8}}
\put(12.6,3.9){\vector(0,-1){2.2}}
\put(12.6,3.9){\vector(1,-1){2.2}}
\put(15.2,3.9){\vector(0,-1){2.2}}
\end{picture}
\end{center}
\vspace{-1cm}
\caption{Graphical representation of temporal dependence of order $q=2$.}
\label{fig:temporal} 
\end{figure}
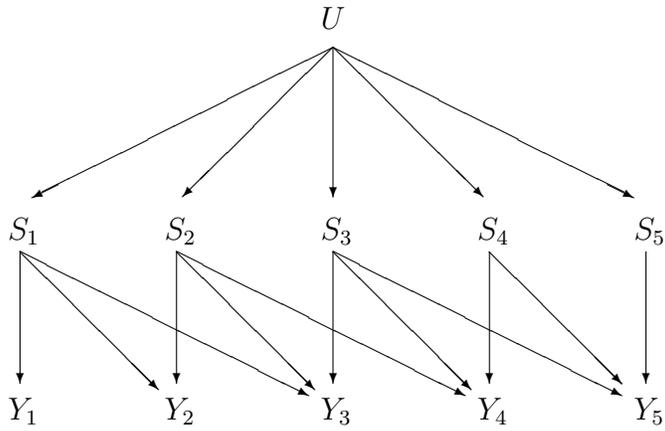

\bigskip
\bigskip
\bigskip

\begin{figure}[h]
\setlength{\unitlength}{0.8cm}
\begin{center}
\begin{picture}(17,8)
\put(14.3,4.4){\oval(4,6)}
\put(14.3,7.4){\line(0,-1){2}}
\put(14.3,1.4){\line(0,1){2}}
\put(12.3,3.4){\line(1,0){4}}
\put(12.3,5.4){\line(1,0){4}}
\put(13,6.3){$Y_{1}$}
\put(15,6.3){$Y_{2}$}
\put(14,4.3){$Y_{3}$}
\put(13,2.3){$Y_{4}$}
\put(15,2.3){$Y_{5}$}

\put(5.8,7.6){$U$} 
\put(6,7.3){\vector(-2,-1){5}}
\put(6,7.3){\vector(-1,-1){2.5}}
\put(6,7.3){\vector(0,-1){2.5}}
\put(6,7.3){\vector(1,-1){2.5}}
\put(6,7.3){\vector(2,-1){5}}
\put(0.6,4.1){$S_1$}
\put(3.2,4.1){$S_2$}
\put(5.8,4.1){$S_3$}
\put(8.4,4.1){$S_4$}
\put(11,4.1){$S_5$}
\put(0.6,1.1){$Y_1$}
\put(3.2,1.1){$Y_2$}
\put(5.8,1.1){$Y_3$}
\put(8.4,1.1){$Y_4$}
\put(11,1.1){$Y_5$}

\put(0.8,3.9){\vector(0,-1){2.2}}
\put(0.8,3.9){\vector(1,-1){2.3}}
\put(0.8,3.9){\vector(2,-1){4.8}}

\put(3.4,3.9){\vector(0,-1){2.2}}
\put(3.4,3.9){\vector(1,-1){2.3}}
\put(3.4,3.9){\vector(-1,-1){2.3}}

\put(6,3.9){\vector(0,-1){2.2}}
\put(6,3.9){\vector(1,-1){2.3}}
\put(6,3.9){\vector(2,-1){4.8}}
\put(6,3.9){\vector(-1,-1){2.3}}
\put(6,3.9){\vector(-2,-1){4.8}}

\put(8.6,3.9){\vector(0,-1){2.2}}
\put(8.6,3.9){\vector(1,-1){2.2}}
\put(8.6,3.9){\vector(-1,-1){2.2}}

\put(11.2,3.9){\vector(0,-1){2.2}}
\put(11.2,3.9){\vector(-1,-1){2.3}}
\put(11.2,3.9){\vector(-2,-1){4.8}}

\end{picture}
\end{center}
\vspace{-1cm}
\caption{Graphical representation of first-order spatial dependence for a five-area region.}
\label{fig:spatial} 
\end{figure}
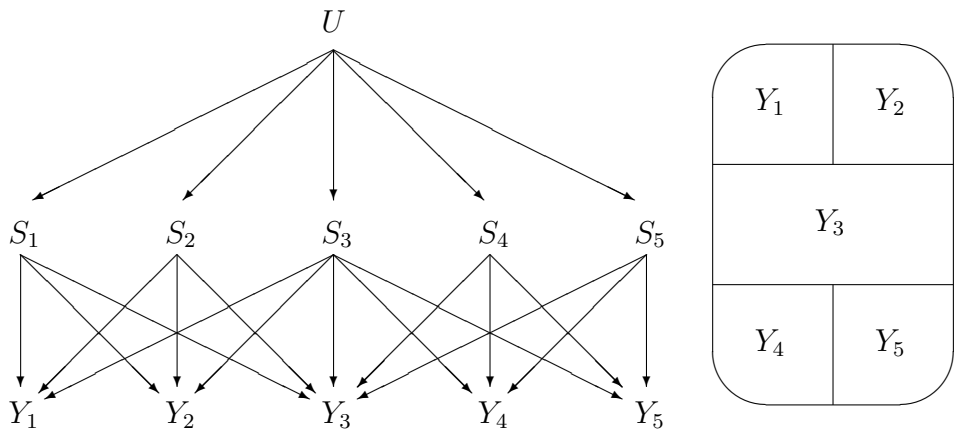

\begin{figure}
\centerline{\includegraphics[scale=0.34]{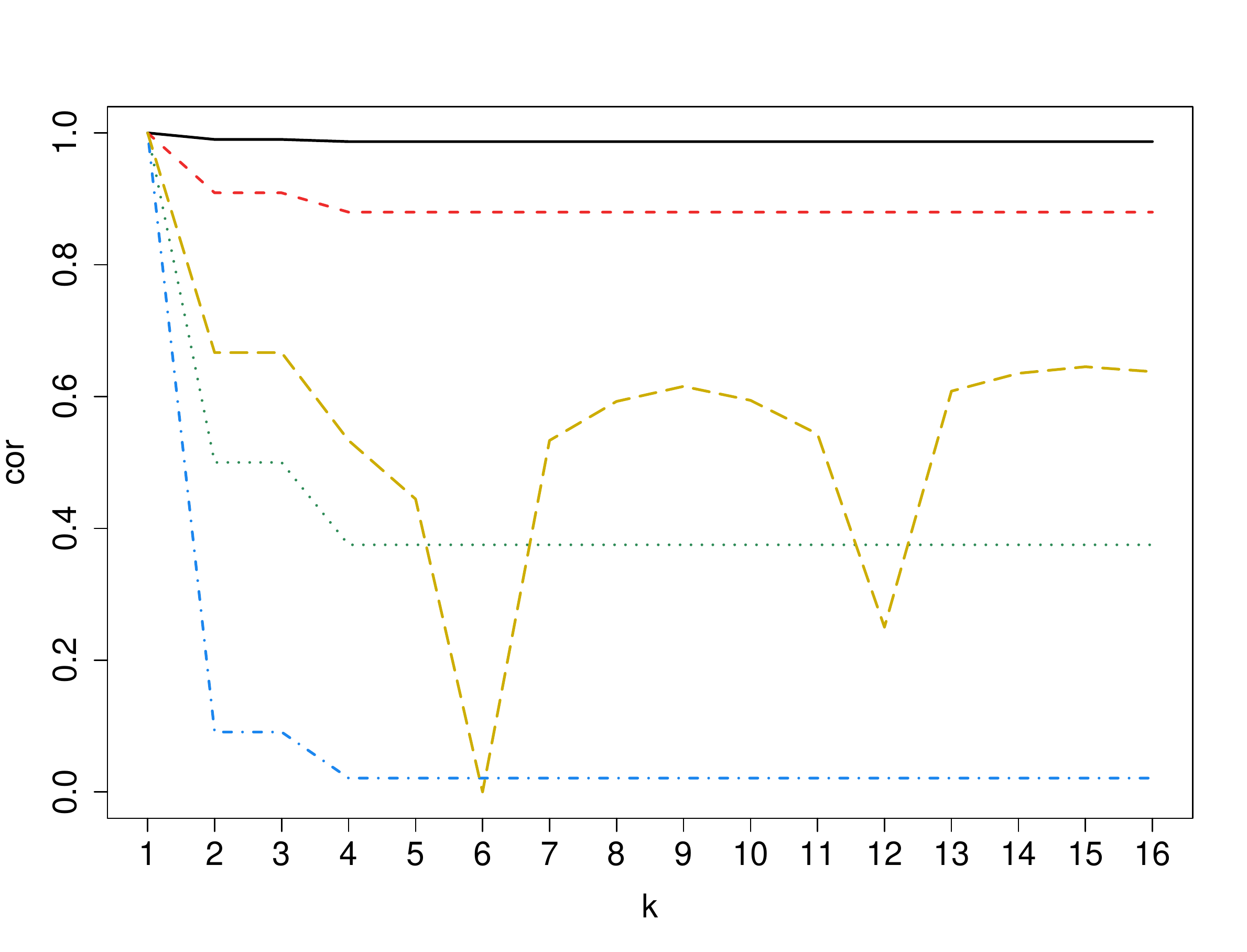}
\includegraphics[scale=0.34]{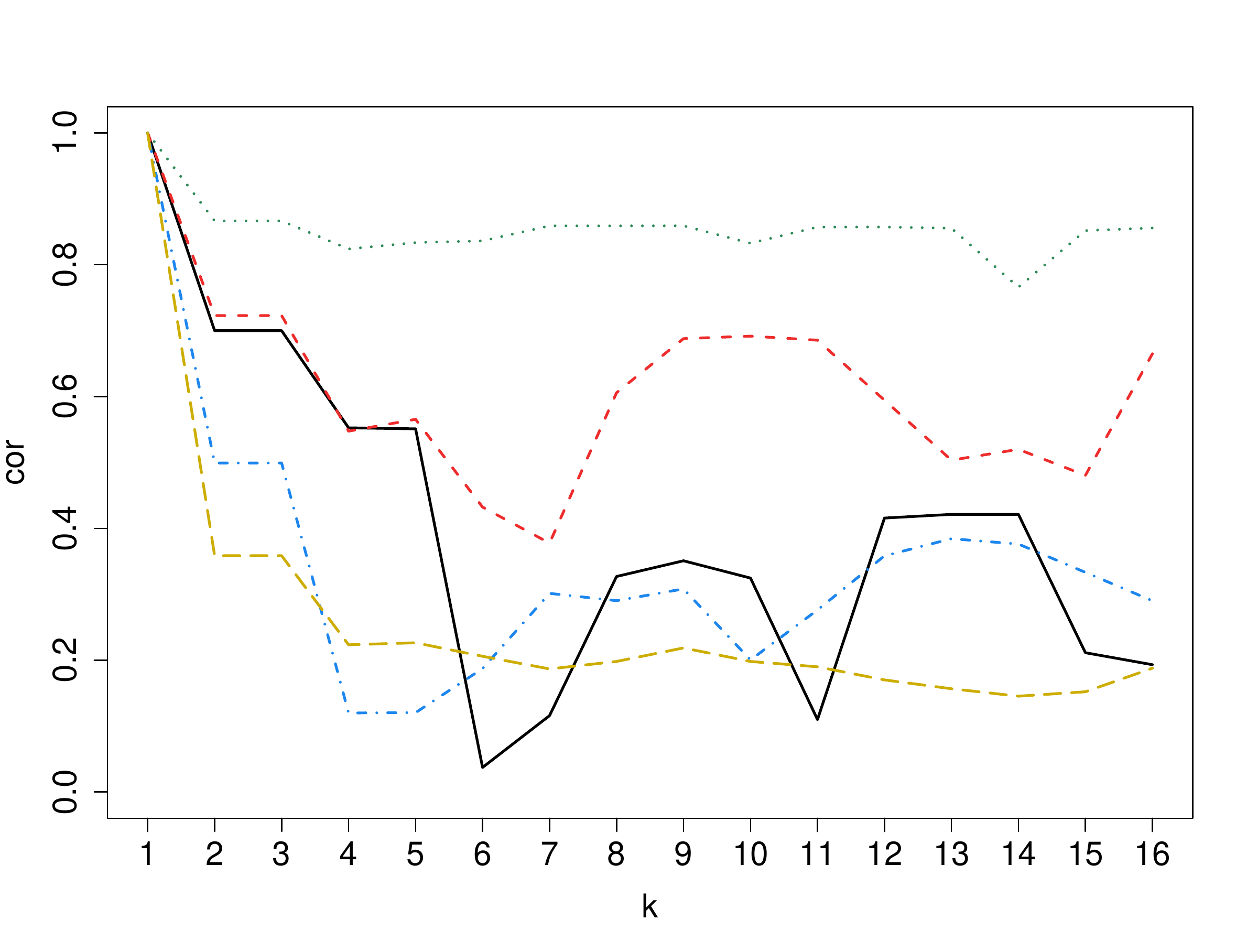}}
\vspace{-2mm}
\caption{{\small Correlation function $\Cor(Y_1,Y_k)$, $k=1,\ldots,16$ for the temporal dependence model of Figure~\ref{fig:temporal}. Left panel: $n_j=1\;\forall j>0$, $n_0=0.01$ (solid line), $n_0=0.1$ (dashed line), $n_0=1$ (dotted line), $n_0=10$ (dotted-dashed line); $\bn=\{1,1,1,0,0,0,2,2,2,0.1,0.1,0.1,5,5,5,1\}$ and $n_0=0.5$ (long dashed line). Right panel: five simulated paths from $n_0\sim\un(0,2)$, $n_j\mid a,b\sim\ga(a,b)$, $j=1,\ldots,16$ with $a\sim\ga(1,1)$ and $b\sim\ga(1,1)$.}}
\label{fig:tcor}
\end{figure}

\begin{figure}
\centerline{\includegraphics[scale=0.34]{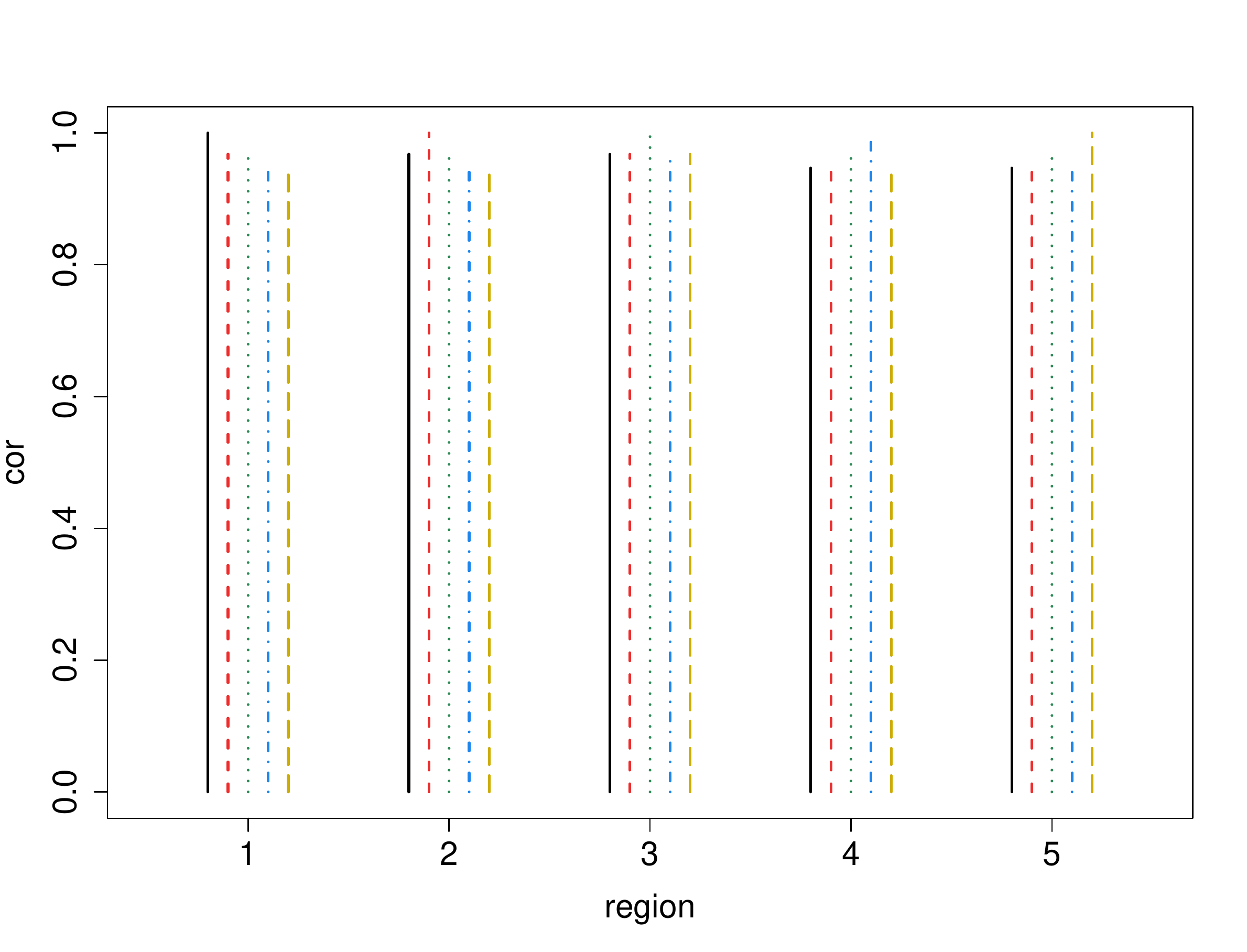}
\includegraphics[scale=0.34]{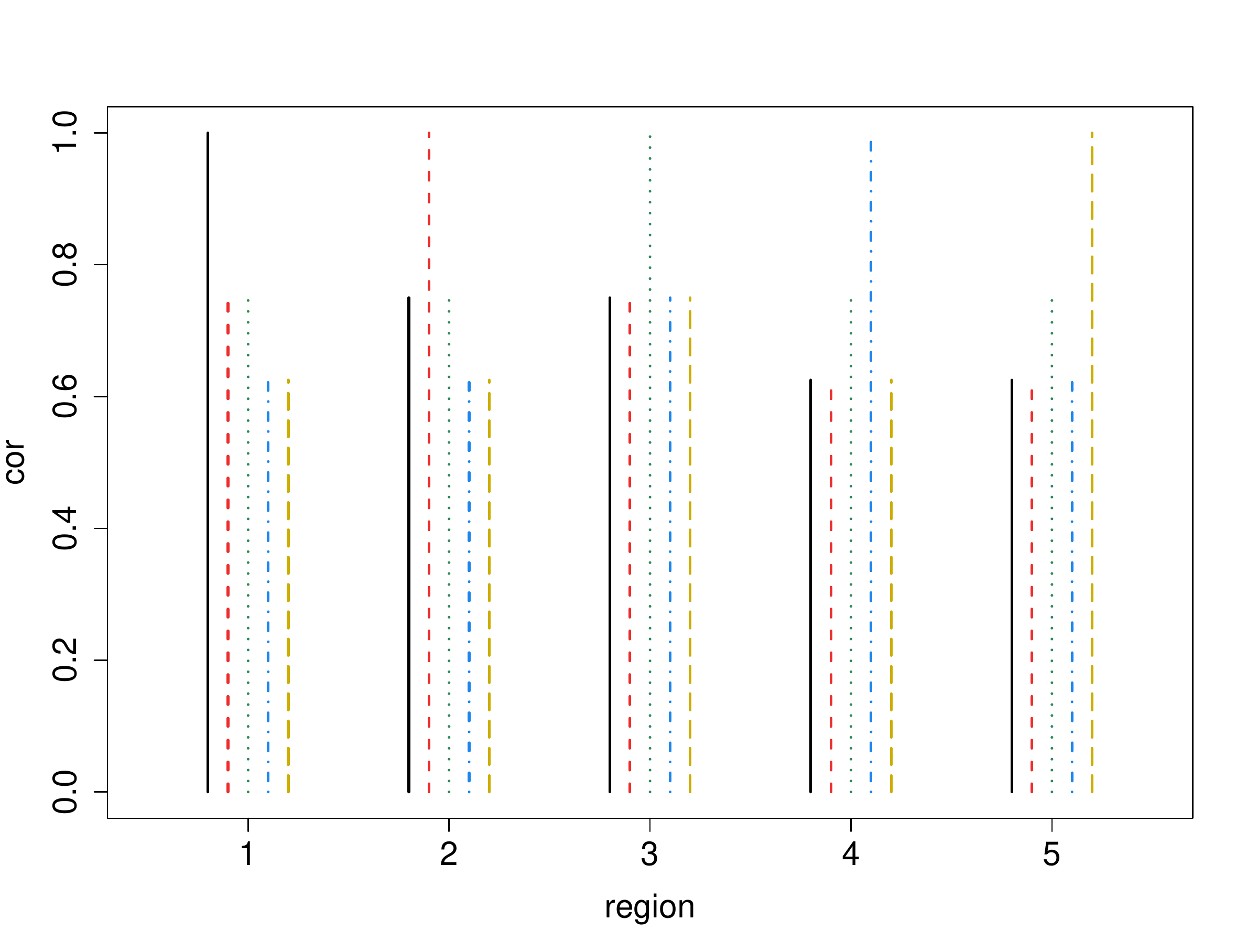}}
\centerline{\includegraphics[scale=0.34]{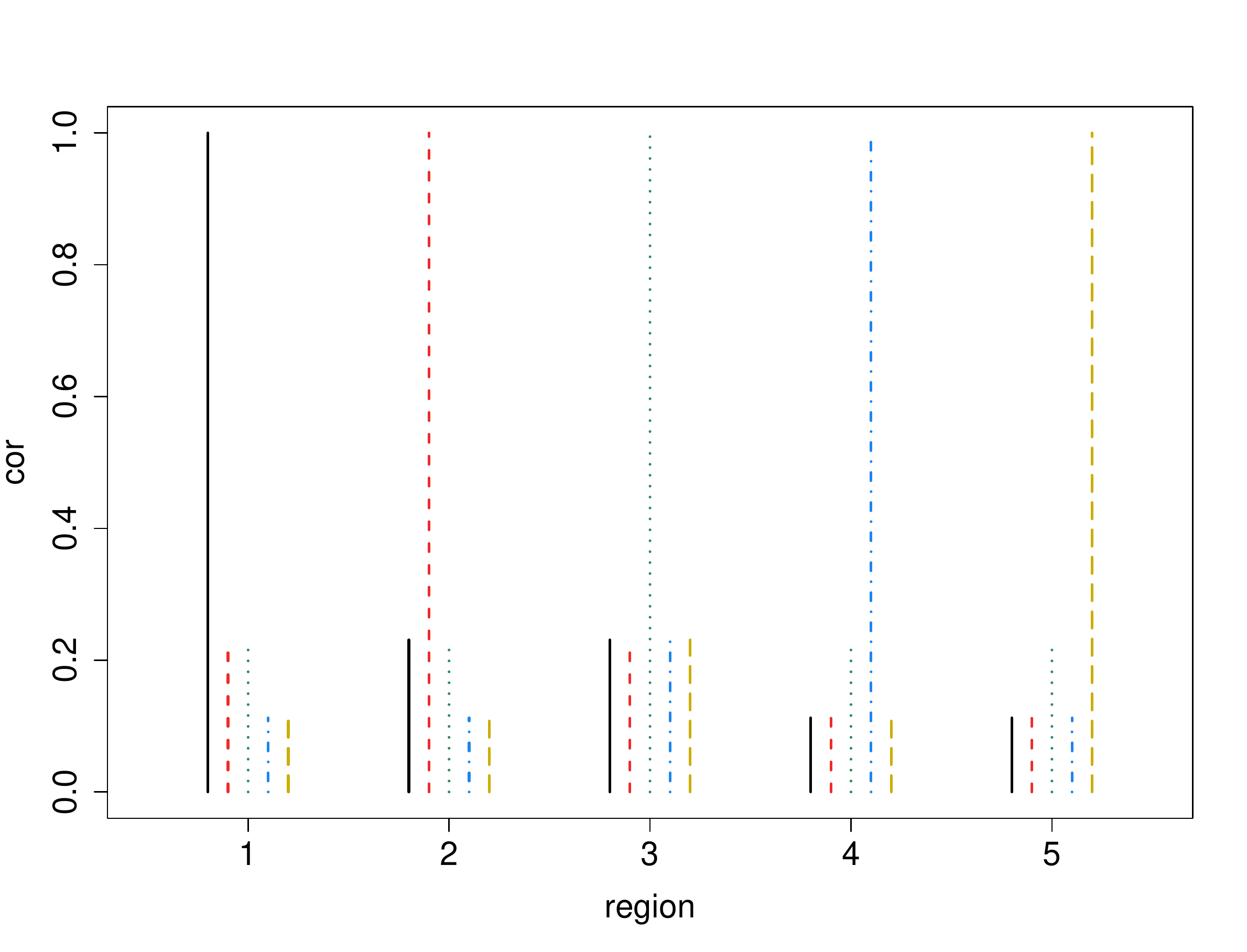}
\includegraphics[scale=0.34]{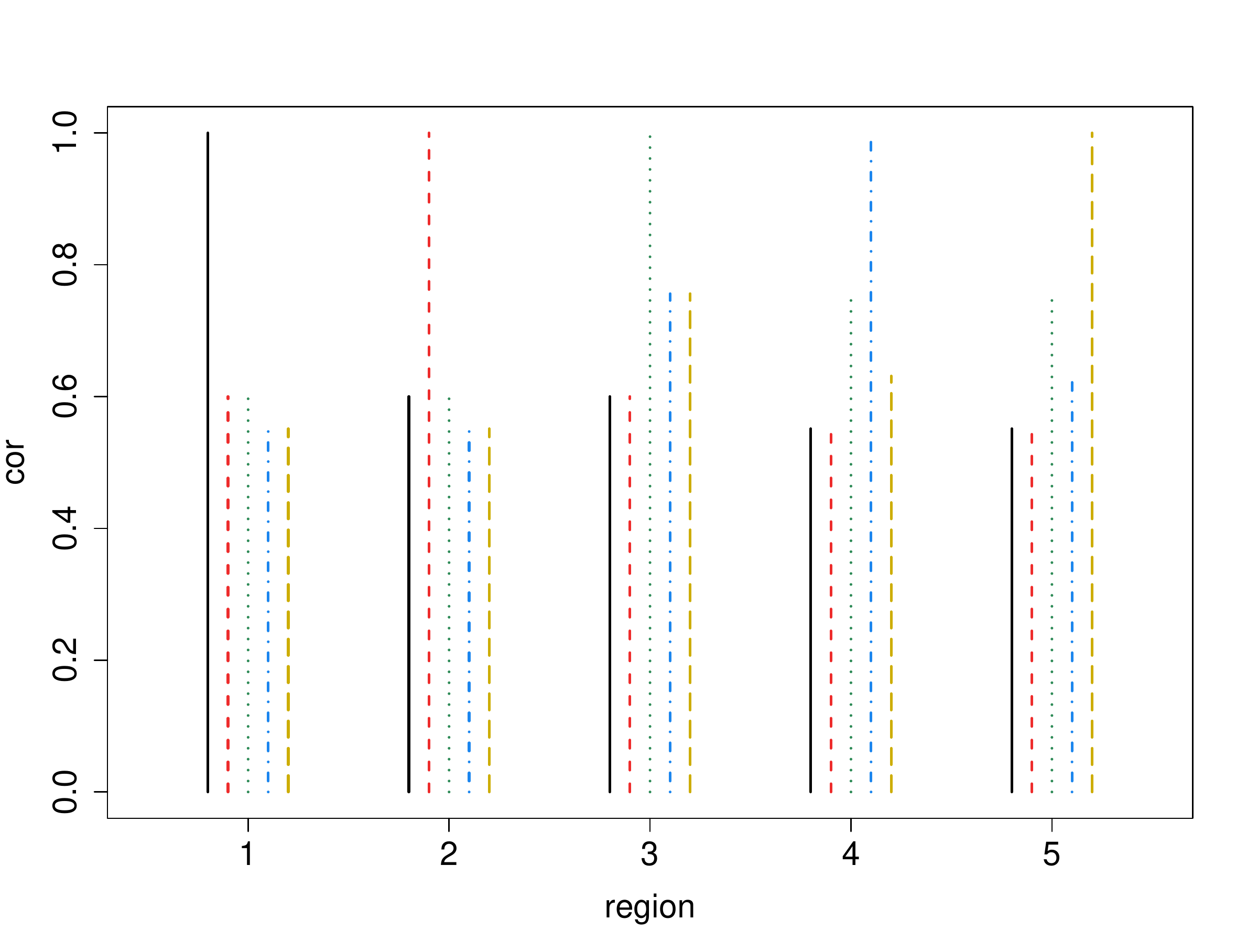}}
\vspace{-2mm}
\caption{{\small Correlation function $\Cor(Y_i,Y_k)$, $i,k=1,\ldots,5$ for the spatial dependence model of Figure~\ref{fig:spatial}. 
$n_j=1\;\forall j>0$, $n_0=0.1$ (top left), $n_0=1$ (top right), $n_0=10$ (bottom left); $\bn=\{0.5,0,1,0.1,2\}$ and $n_0=1$ (bottom right).}}
\label{fig:scor1}
\end{figure}

\begin{figure}
\centerline{\includegraphics[scale=0.34]{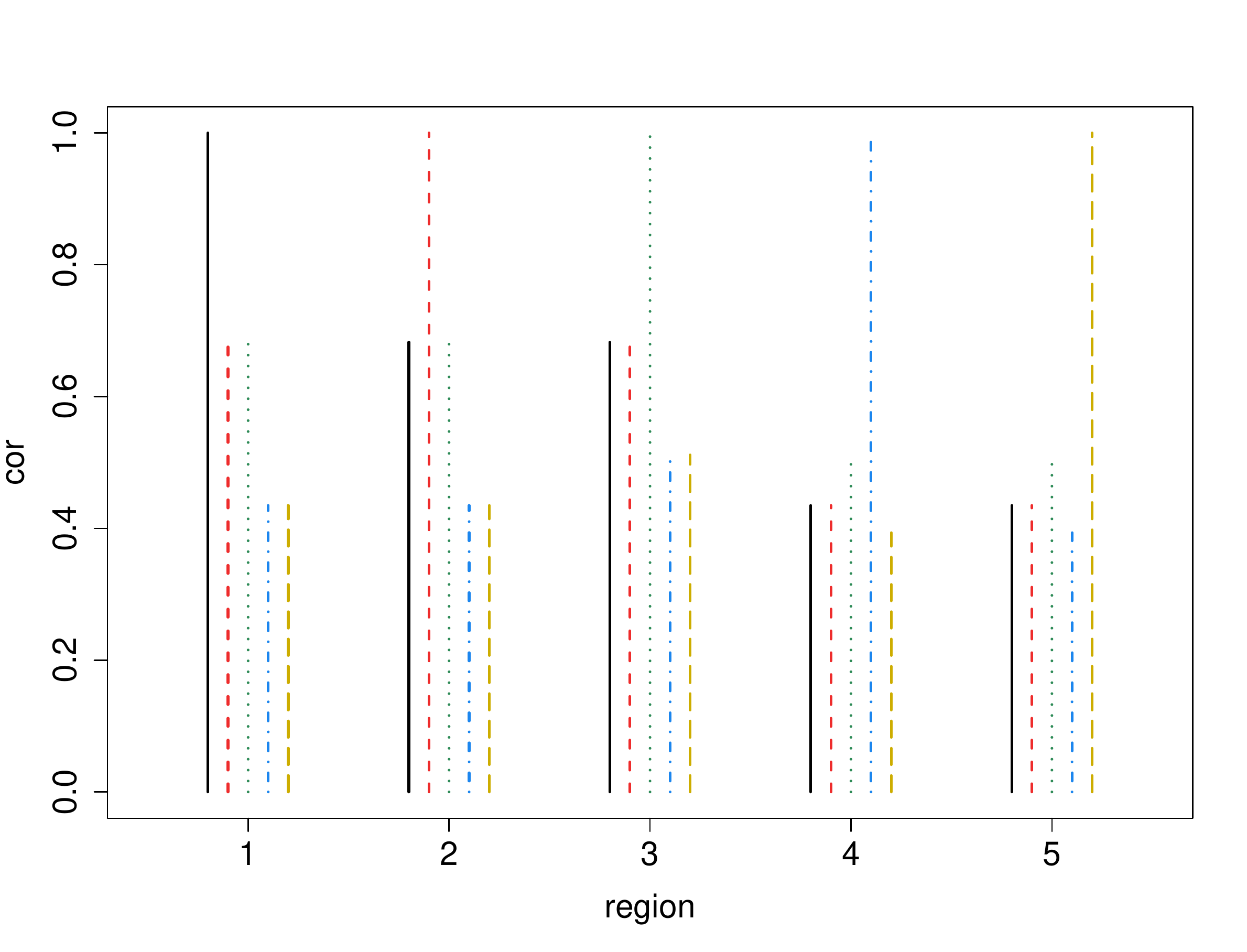}
\includegraphics[scale=0.34]{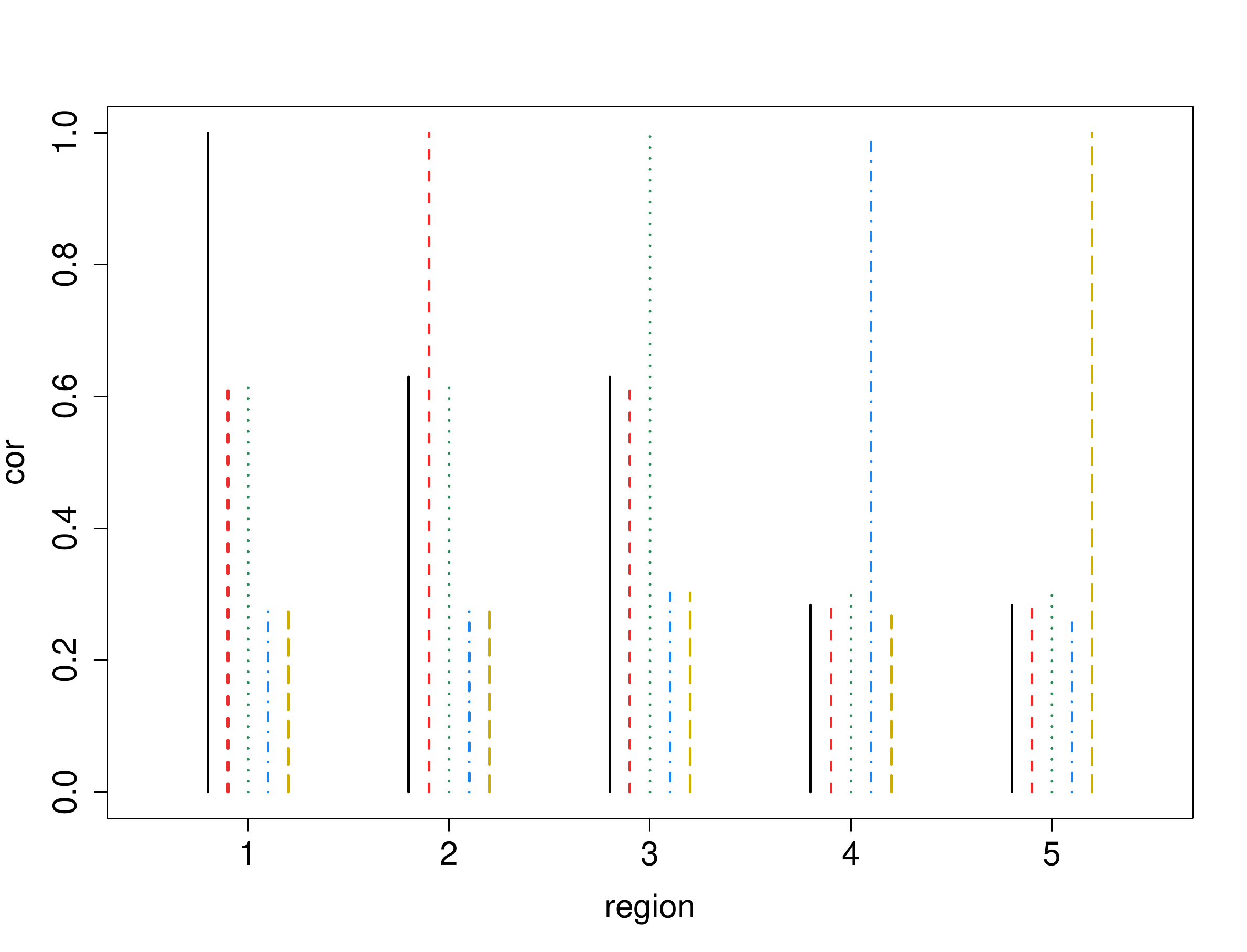}}
\centerline{\includegraphics[scale=0.34]{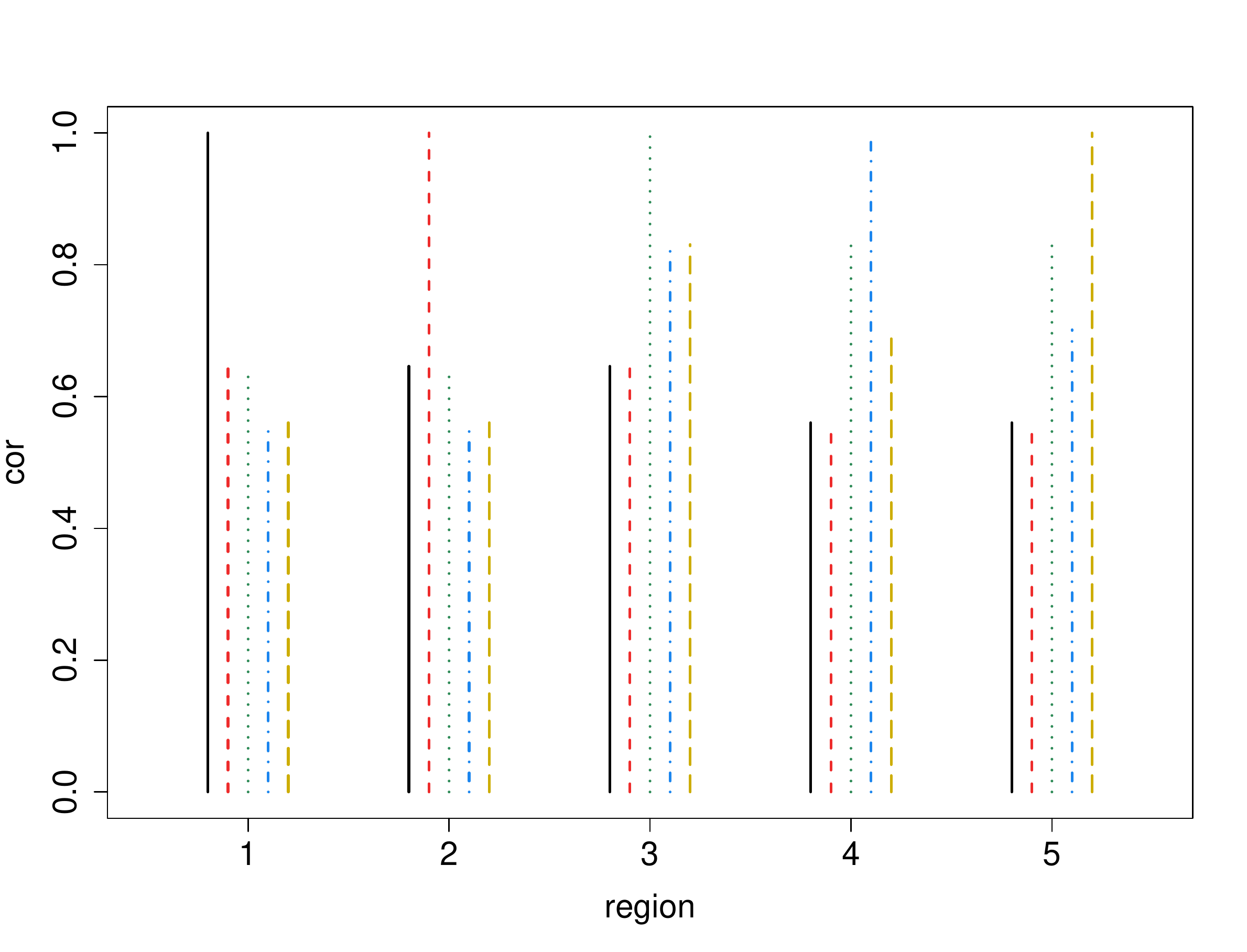}
\includegraphics[scale=0.34]{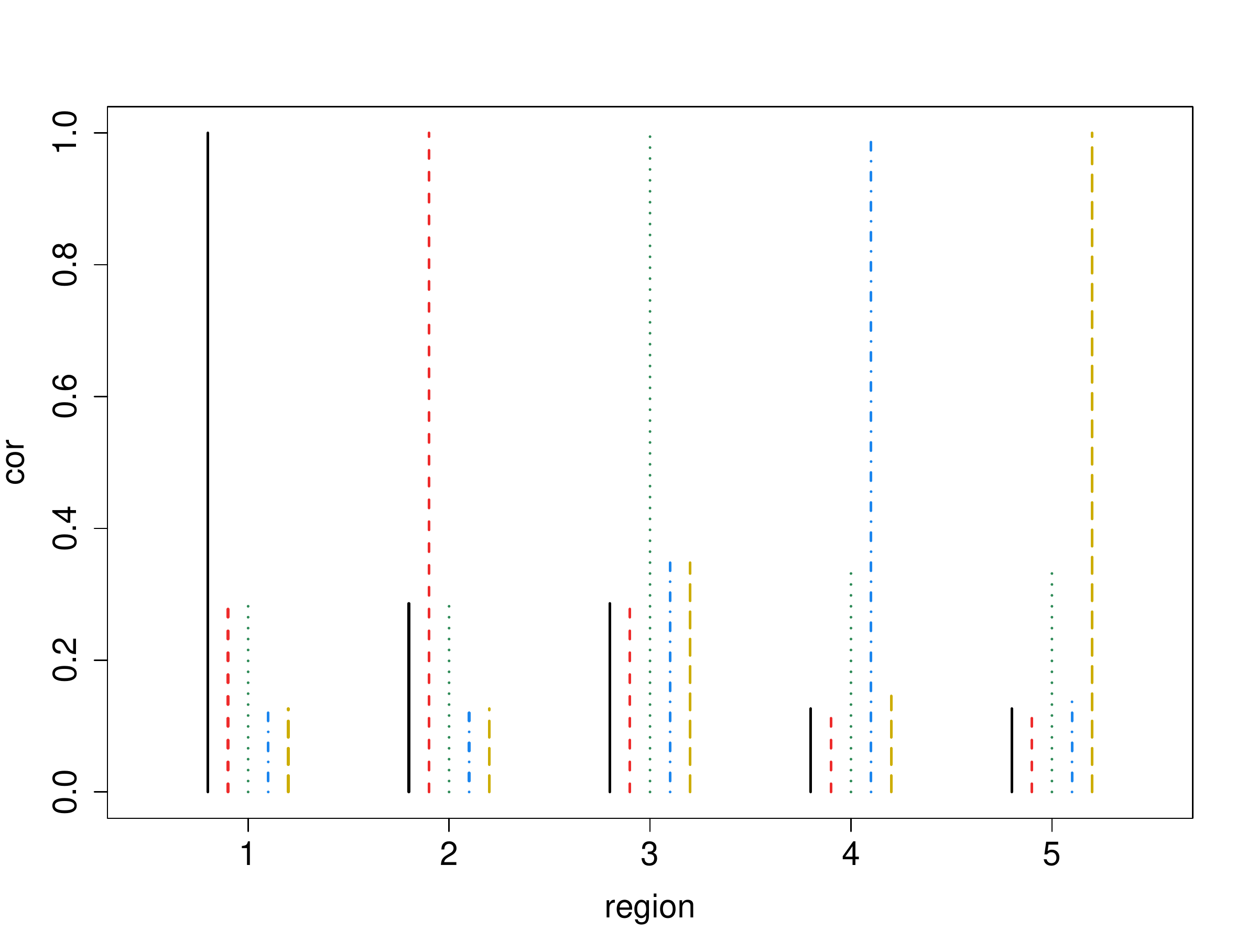}}
\vspace{-2mm}
\caption{{\small Correlation function $\Cor(Y_i,Y_k)$, $i,k=1,\ldots,5$ for the spatial dependence model of Figure~\ref{fig:spatial}. Four simulated scenarios from $n_0\sim\un(0,2)$, $n_j\mid a,b\sim\ga(a,b)$, $j=1,\ldots,5$ with $a \sim \ga(1,1)$ and $b \sim \ga(1,1)$.}}
\label{fig:scor2}
\end{figure}

\begin{figure}
\centerline{\includegraphics[scale=0.27]{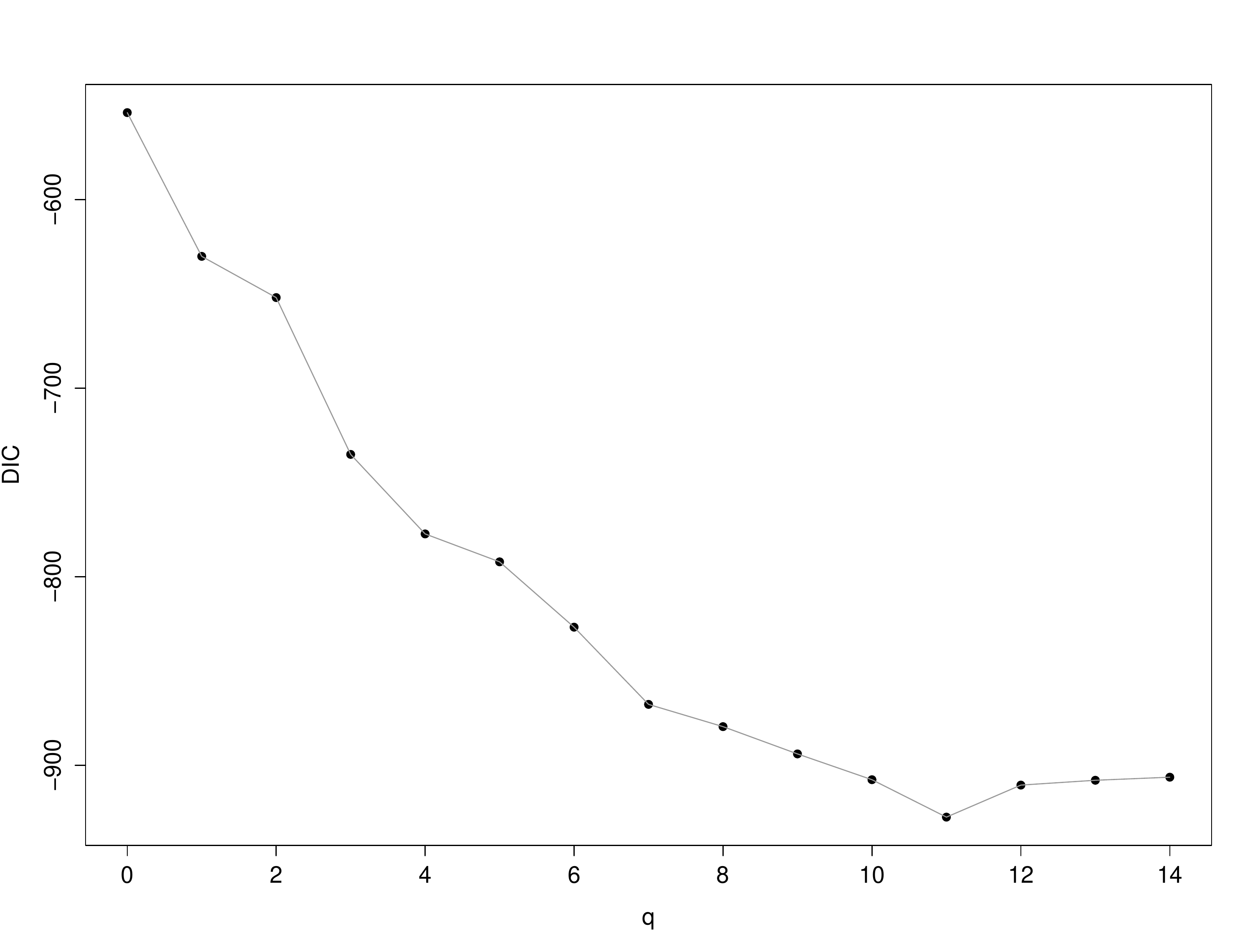}
\includegraphics[scale=0.27]{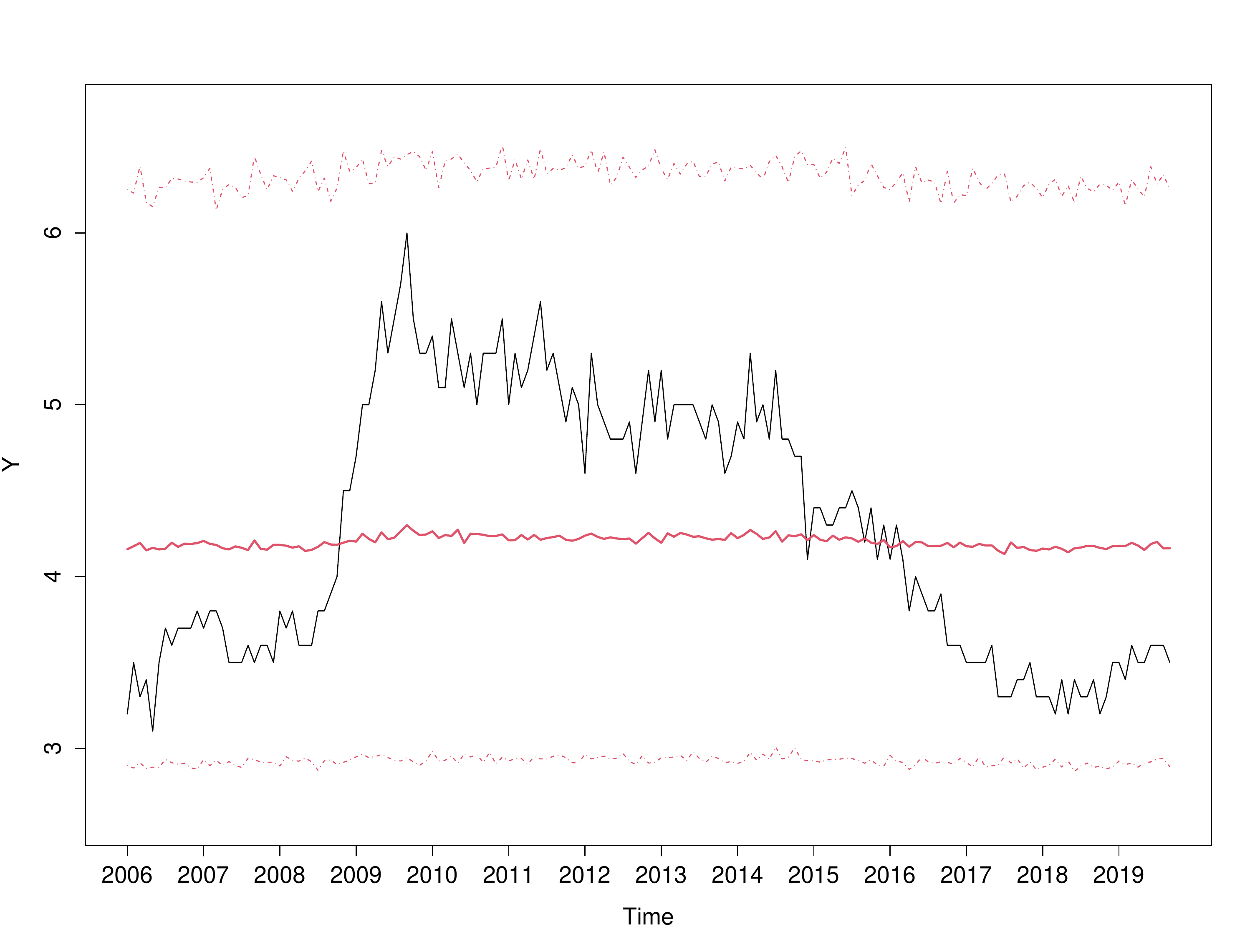}}
\centerline{\includegraphics[scale=0.27]{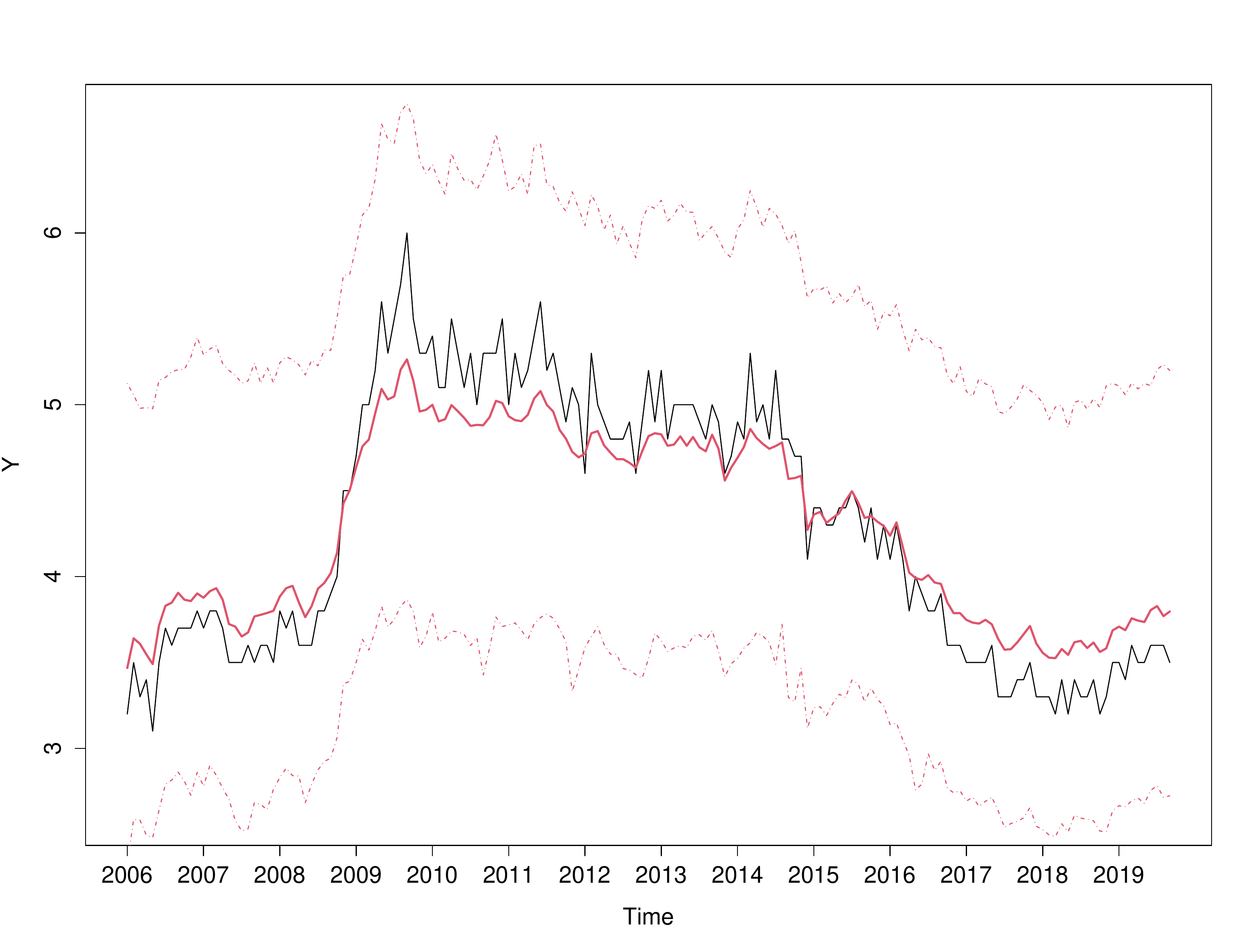}
\includegraphics[scale=0.27]{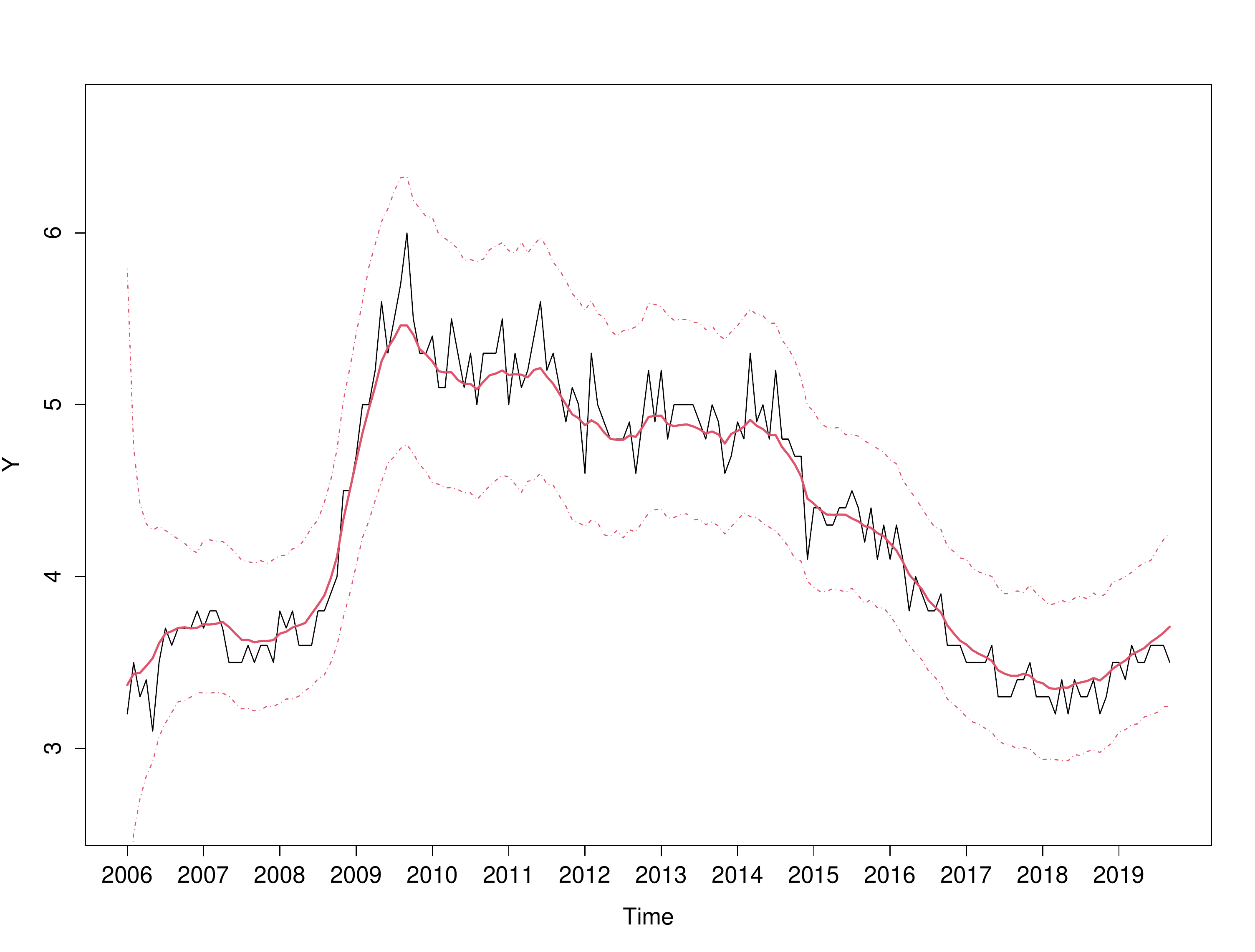}}
\vspace{-2mm}
\caption{{\small Monthly unemployment rates. DIC values (top-left panel). Model fit for $q=0$ (top-right), $q=1$ (bottom-left) and $q=13$ (bottom-right). Original data (solid line), point predictions (thick solid line) and 95\% credible intervals (dotted lines).}}
\label{fig:dataT}
\end{figure}

\begin{figure}
\vspace{-2cm}
\centerline{\includegraphics[scale=0.55]{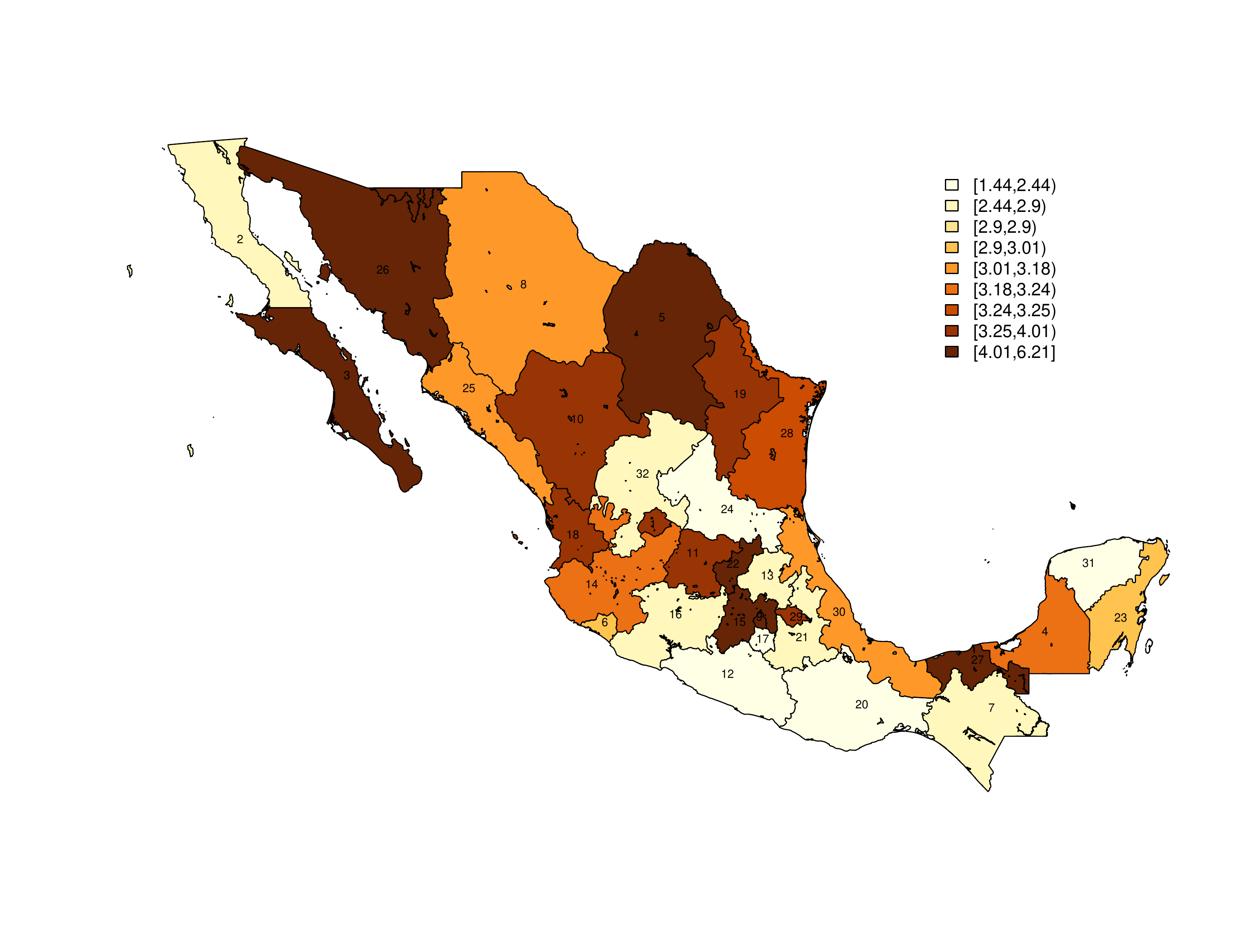}}
\vspace{-1.5cm}
\centerline{\includegraphics[scale=0.55]{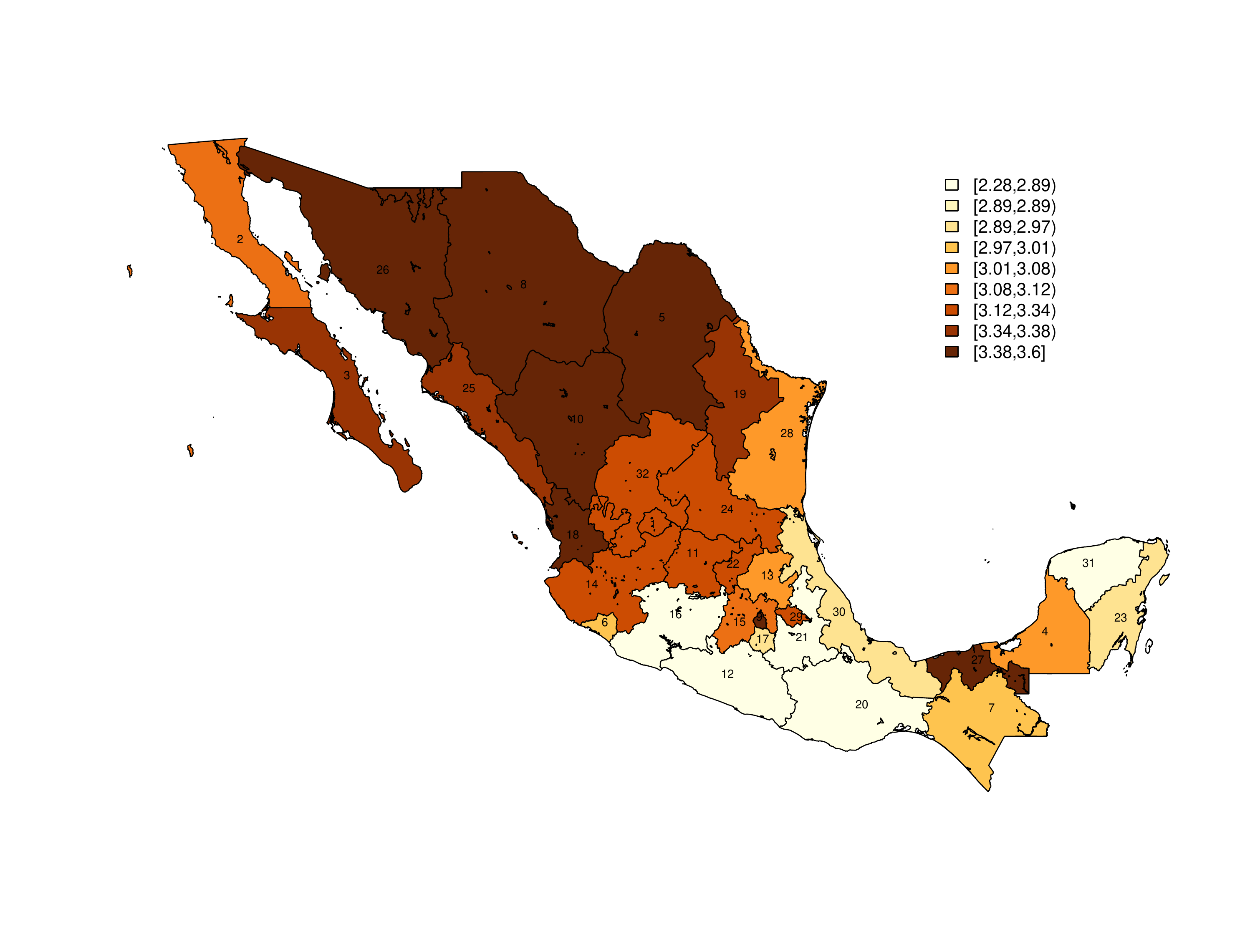}}
\vspace{-1.5cm}
\caption{{\small Unemployment rates for the 32 Mexican States. Observed data (top), model fit (bottom).}}
\label{fig:dataS}
\end{figure}

\end{document}